\documentclass[pdftex]{amsart} 

\usepackage{mathtools}

\usepackage{array,rotating,pdflscape,
  amsfonts,amssymb,amsmath,latexsym,upref,enumerate,
  cool}
\usepackage[margin=2.0cm]{geometry}
\usepackage{verbatim} 
\usepackage[english]{babel}
\usepackage[latin1]{inputenc} 
\usepackage{color,tikz,pgf}

\usetikzlibrary{shapes, intersections, calc, patterns, positioning,
  matrix, arrows, fit, backgrounds, 
  decorations.pathmorphing, decorations.markings}

\newtheorem{lemma}[equation]{Lemma}
\newtheorem{prop}[equation]{Proposition}
\newtheorem{thm}[equation]{Theorem}
\newtheorem{cor}[equation]{Corollary}

\newtheorem{defn}[equation]{Definition}

\theoremstyle{definition}
\newtheorem{exmp}[equation]{Example}

\newtheorem{rmk}[equation]{Remark}

\numberwithin{equation}{section}

\newcommand{\Z}{\mathbf{Z}}

\newcommand{\Q}{\mathbf{Q}}
\newcommand{\C}{\mathbf{C}}
\newcommand{\F}{\mathbf{F}}
\newcommand{\R}{\mathbf{R}}

\newcommand{\Gal}[2]{\operatorname{Gal}(#1,#2)}

\newcommand{\Hom}[3][{}]{\operatorname{Hom}_{#1}(#2,#3)}

\newcommand{\Irr}[2]{\operatorname{Irr}_{#1}(#2)}

\newcommand{\ord}[2]{\operatorname{ord}_{#1}(#2)}

\newcommand{\GL}[3][{+}]{\operatorname{GL}^{#1}_{#2}(\F_{#3})}
\newcommand{\GU}[2]{\operatorname{GL}^-_{#1}(\F_{#2})}
\newcommand{\gl}{\GL nq}
\newcommand{\gu}{\GU nq}
\newcommand{\Li}{\operatorname{L}}

\newcommand{\SU}[2]{\operatorname{SL}^-_{#1}(\F_{#2})}

\newcommand{\cat}[3]{\mathcal{#1}_{#2}^{#3}}  
  
\newcommand{\m}{morphism}
\newcommand{\pol}{polynomial}

\newcommand{\gen}[1]{{\langle}#1{\rangle}}

\newcommand{\Mb}{M\"obius}
\newcommand{\Euc}{Euler characteristic}
\newcommand{\rchi}{\widetilde{\chi}}

\newcommand{\B}{\operatorname{B}}

\newcommand{\im}{\operatorname{im}}

\newcommand{\mynote}[1]{\noindent{\textcolor{red}{\textbf{[#1]}}}}

\newcommand{\diag}{\operatorname{diag}}

\newcommand{\FGL}[2]{\operatorname{FGL}^{#1}_{#2}}
\newcommand{\GGL}[2]{\operatorname{GGL}^{#1}_{#2}}

\newcommand{\IM}[2]{\operatorname{IM}_{#1}(#2)}
\newcommand{\SDIM}[3]{\operatorname{SDIM}^{#1}_{#2}(#3)}

\title{Equivariant Euler characteristics of unitary buildings}

\author{Jesper M.~M\o ller}
\address{Institut for Matematiske Fag\\
  Universitetsparken 5\\
  DK--2100 K\o benhavn}
\email{moller@math.ku.dk}
\urladdr{htpp://www.math.ku.dk/~moller}

\thanks{Supported by the Danish National Research Foundation through
  the Centre for Symmetry and Deformation (DNRF92)}

\usepackage[bookmarks=true,bookmarksopen=false]{hyperref}
\hypersetup{
   pdftitle = {},
   pdfauthor = {Jesper Michael MÃÂÃÂÃÂÃÂ¸ller},
   pdfpagemode = {UseOutlines},
   pdfstartview = {FitH},
   pdfborder = {0 0 0},
   backref = {true},
   colorlinks = {true},
   urlcolor = {blue},
   citecolor = {blue},
   linkcolor = {blue},
   pdftoolbar = {true},
}

\subjclass[2010]{05E18, 06A07} \keywords{Equivariant \Euc, totally
  isotropic subspace, general unitary group over a finite field,
  generating function, irreducible polynomial}

\begin{document}
\date{\today}
\begin{abstract}
  The ($p$-primary) equivariant \Euc s of the buildings for the general unitary groups
  over finite fields are determined.
\end{abstract}
\maketitle
\tableofcontents

\section{Introduction}
\label{sec:intro}


Let $G$ be a finite group, $\Pi$ a finite $G$-poset, and $r \geq 1$ a
natural number.  The \emph{$r$th equivariant reduced \Euc\/} of the
$G$-poset $\Pi$ as defined by Atiyah and Segal \cite{atiyah&segal89}
is the normalised sum
\begin{equation}\label{defn:rchiPiG}
  \rchi_r(\Pi,G) = \frac{1}{|G|} \sum_{X \in \Hom{\Z^r}G} \rchi(C_{\Pi}(X(\Z^r))
\end{equation}
of the reduced \Euc s of the $X(\Z^r)$-fixed $\Pi$-subposets,
$C_{\Pi}(X(\Z^r))$, as $X$ ranges over the set of all homo\m s of
$\Z^r$ to $G$. For example, when $G$ acts trivially on $\Pi$,
$\rchi_r(\Pi,G) = \rchi(\Pi) |\Hom{\Z^r}G|/|G|$ is
the reduced \Euc\ of $\Pi$ times
the number of conjugacy classes
of commuting $(r-1)$-tuples of elements of $G$ \cite[Lemma
4.13]{HKR2000}.
Here are three more examples of already known equivariant reduced \Euc s:
\begin{enumerate}
\item The symmetric group $\Sigma_n$ acts on the poset
  $B(n)^*$ of non-extreme subsets of the $n$-set.
The generating function for the $(r+1)$th  equivariant reduced \Euc s of the $\Sigma_n$-posets $B(n)^*$ is
  \begin{equation*}
  1+\sum_{n \geq 1}\rchi_{r+1}(B(n)^*,\Sigma_n)x^n = \prod_{n \geq 1} (1-x^n)^{\lambda_r(n)}
\end{equation*}
where $\lambda_r(n)$ is the number of index $r$-subgroups of $\Z^n$
(\cite[Theorem~2.1]{bryan-fulman98}  with $M=\{0,1\}$).
\item The symmetric group $\Sigma_n$ acts on the poset
  $\Pi(n)^*$ of non-extreme partitions of the $n$-set. The equivariant reduced
  \Euc s  of the $\Sigma_n$-posets $\Pi(n)^*$ are
  \begin{equation*}
       \rchi_r(\Pi(n)^*, \Sigma_n) = \frac{1}{n}c_r(n)    
  \end{equation*}
where $c_r$ is the sequence with Dirichlet convolution $(c_r \ast \lambda_r)(n) = (-1)^{n+1}$
\cite[Theorem~1.3, Corollary~1.4]{jmm:partposet2017}.
\item The general linear group $\GL nq$ acts on the poset $\Li_n^+(\F_q)^*$ of non-extreme subspaces of the $n$-dimensional vector space over the field $\F_q$ of prime power order $q$. The generating function
   for the $(r+1)$th equivariant reduced \Euc s of the $\GL nq$-posets $\Li_n^+(\F_q)^*$ is
\begin{equation*}
    1 + \sum_{n \geq 0} \rchi_{r+1}(\Li_n^+(\F_q), \GL nq) x^n = \prod_{0 \leq j \leq r}
  (1-q^jx)^{(-1)^{r-j}\binom{r}{j}}
\end{equation*}
according to \cite[Theorem~1.4]{jmm:eulergl+}.
\end{enumerate}

In this article we shall consider the general unitary group
$\gu$, the isometry group of the unitary
$n$-geometry over the field $\F_{q^2}$, acting on the poset
$ \Li_n^-(\F_q)^* =\{ 0 \subsetneq U \subsetneq \F_{q^2}^n \mid U
\subseteq U^\perp \} $ of nontrivial totally isotropic subspaces.
(See Section~\ref{sec:building} for more details.)
We now emphasise the definition of the equivariant reduced \Euc s in this particular case and
proceed to present the main results of this paper.

\begin{defn}\cite{atiyah&segal89}\label{defn:chir}
  The $r$th equivariant reduced \Euc\ of the $\gu$-poset
  ${\Li}_n^-(\F_q)^*$ is the normalised sum
\begin{equation*}
 \rchi_r({\Li}_n^-(\F_q)^*,\gu) = \frac{1}{|\gu|}
   \sum_{X \in \Hom {\Z^r}{\gu}} \rchi(C_{\Li_n^-(\F_q)^*}(X(\Z^r))) 
\end{equation*}
of the \Euc s of the induced subposets $C_{\Li_n^-(\F_q)^*}(X(\Z^r))$ of
$X(\Z^r)$-invariant subspaces as $X$ ranges over all homo\m s of the
free abelian group $\Z^r$ on $r$ generators into the general unitary
group.
\end{defn}

The generating function for the {\em negative\/}  $r$th equivariant reduced \Euc s is the power series
\begin{equation}\label{eq:genfct}
  \FGL -r(q,x) = 1 - \sum_{n \geq 1} \rchi_r(\GU nq)x^n
\end{equation}
with coefficients in the ring of integral polynomials in $q$.  (The
shortened notation $\rchi_r(\gu)$ is and will be used for the $r$th
equivariant reduced \Euc\ $\rchi_r({\Li}_n^-(\F_q)^*,\gu)$ of
Definition~\ref{defn:chir}.)

\begin{thm}\label{thm:main}
  $\displaystyle \FGL -{r+1}(q,x) = \prod_{0 \leq j \leq r}
  (1+(-1)^{r-j}q^{j}x)^{(-1)^{r-j} \binom rj}$ for all $r \geq 0$. 
  \end{thm}

The first few instances of the generating function 
\begin{equation*}
  \FGL -{r+1}(q,x) = 
  \prod_{0 \leq j \leq r} (1+(-1)^{r-j}q^{j}x)^{(-1)^{r-j} \binom rj}
  =
  \frac{\displaystyle \prod_{j \equiv r \bmod 2}
    (1+q^{j}x)^{\binom rj}}
  {\displaystyle \prod_{j \not\equiv r \bmod 2}
    (1-q^{j}x)^{\binom rj}}
\end{equation*}
are
\begin{equation*}
  1+x, \quad
  \frac{1+qx}{1-x}, \quad
  \frac{(1+q^2x)(1+x)}{(1-qx)^2}, \quad
  \frac{(1+q^3x)(1+qx)^3}{(1-q^2x)^3(1-x)}, \quad
  \frac{(1+q^4x)(1+q^2x)^6(1+x)}{(1-q^3x)^4(1-qx)^4}
\end{equation*}
for $r+1=1,2,3,4$. In particular, $-\rchi_2(\GU nq) = q+1$ and
$-\rchi_3(\GU nq) = nq^{n-1}(q+1)^2$ for $n \geq 1$.

The proofs of Theorem~\ref{thm:main} and its corollary below are in Section~\ref{sec:computation}.


\begin{cor}\label{cor:expform}
  $\displaystyle
  \FGL -{r+1}(q,x)
 =\exp\big( -\sum_{n \geq 1} (-1)^n
  (q^n-(-1)^n)^r \frac{x^n}{n} \big)
$ for all $r \geq 0$.
\end{cor}

We also consider, for any prime $p$, the $p$-primary equivariant  reduced \Euc s, $\rchi_r(p,\GU nq)$, 
for the $\GU nq$-poset $\Li_n^-(\F_q)^*$ (Definition~\ref{defn:pprimeuc})
as defined by Tamanoi \cite{tamanoi2001}.
It turns out that the $r$th $p$-primary generating
function at $q$, the generating function $\FGL -r(p,q,x)$ for the
negative $r$th $p$-primary equivariant reduced \Euc s 
\eqref{eq:primgenfct}, is in some sense a $p$-local version
of the exponential form of $\FGL -{r}(q,x)$ from
Corollary~\ref{cor:expform}. (We write $n_p$ for the $p$-part of the
natural number $n$.)

\begin{thm} \label{thm:mainprim}
$\displaystyle
    \FGL{-}{r+1}(p,q,x)
      =\exp\big( -\sum_{n \geq 1}(-1)^{n}(q^n-(-1)^{n})_p^r \frac{x^n}{n} \big)
 $
  for all $r \geq 0$.
\end{thm}

  The infinite product
  expansions 
  of the absolute and the $p$-primary generating functions 
  \begin{alignat*}{3} 
    &\FGL -{r+1}(q,x)  = \prod_{n \geq 1}
    (1-x^n)^{a_{r+1}^-(q,n)}
    && \qquad
    &&a_{r+1}^-(q,n) = \frac{1}{n} \sum_{d \mid n} (-1)^d \mu(n/d)
    (q^d-(-1)^d)^r \\
    &\FGL -{r+1}(p,q,x)  = \prod_{n \geq 1}
    (1-x^n)^{a_{r+1}^-(p,q,n)}
    && \qquad
    &&a_{r+1}^-(p,q,n) = \frac{1}{n} \sum_{d \mid n} (-1)^d \mu(n/d)
    (q^d-(-1)^d)_p^r
  \end{alignat*}
  follow immediately from Theorems~\ref{thm:main} and
  \ref{thm:mainprim} using the elementary \cite[Lemma~3.7]{jmm:eulergl+}.

 More explicitly, the equivariant \Euc s and the $p$-primary \Euc s of the general unitary group are 
  \begin{equation*}
  -\rchi_{r+1}(\GU nq) = 
  \frac{1}{|W_n|} \sum_{w \in W_n} \det(w) \det(q + w)^r, \quad
  -\rchi_{r+1}(p,\GU nq) = 
  \frac{1}{|W_n|} \sum_{w \in W_n} \det(w) \det(q + w)^r_p
\end{equation*}
where the sum ranges over the elements of the standard $n$-dimensional integral permutation representation $W_n$ of the symmetric group $\Sigma_n$ (Propositions~\ref{prop:altglpm}, \ref{prop:altglpmpprim}). 

This paper is organised as follows. Section~\ref{sec:building} describes the general unitary group $\GU nq$ as a subgroup of $\GL n{q^2}$. Characteristic \pol s for elements of $\GU nq$ are self-dual and Section~\ref{sec:self-dual-polyn} is an exposition of the combinatorics of self-dual irreducible monic \pol s over $\F_{q^2}$. The main result of Section~\ref{sec:equiv-reduc-euc} is the product formula of Lemma~\ref{lemma:chiproduct} for equivariant \Euc s.  Section~\ref{sec:semi-simple-unitary} establishes the key ingredients in the proof of Theorem~\ref{thm:main}: a vanishing result (Lemma~\ref{lemma:contractability})
  and a recursion formula (Lemma~\ref{lemma:recur}).
   Theorem~\ref{thm:main} and Corollary~\ref{cor:expform} are proved in Section~\ref{sec:computation}.
  This section also connects the equivariant \Euc s of the general unitary group to representation theory  and algebraic geometry:  
  Remark~\ref{rmk:KRC} explains the role of the second equivariant \Euc\ $\rchi_2(\GU nq)$ in the
  Kn\"orr-Robinson conjecture for $\GU nq$ at the defining characteristic \cite{knorr_robinson:89},
  and
  Subsection~\ref{sec:conn-hasse-weil} points out a curious coincidence with
  Hasse--Weil zeta functions of supersingular elliptic curves over $\F_{q^2}$. 
  In Section~\ref{sec:power-seri-expans-1} we develop more explicit expressions for the equivariant \Euc s.
  (The formulas of Proposition~\ref{prop:altglpm} may indicate a general description of the equivariant \Euc s of finite groups of Lie type.)  In Section~\ref{sec:mult-t_a-transf} we
  shortly review  the $S$-transform and use it to (re)prove  \pol\ identities associated to partitions.
  Section~\ref{sec:p-prim-equiv} discusses $p$-primary equivariant reduced \Euc s of general unitary groups for a given prime $p$. The corresponding {\em unreduced\/} \Euc s can be interpreted as \Euc s computed in Morava $K$-theories at $p$ of the homotopy orbit spaces $\B\!\Li^-_n(\F_q)^*_{h\!\gu}$ for the action of $\GU nq$ on the classifying space of the poset $\Li_n^-(\F_q)^*$. The proof of Theorem~\ref{thm:mainprim} together with more explicit expressions for the $p$-primary equivariant \Euc s $\rchi_{r}(p,\GU nq)$ can be found here.

The following notation will be used in this paper 
in addition to notation related to multisets introduced at the beginning of Subsection~\ref{sec:power-seri-expans-1}:

\begin{tabular}[h]{c|l}
  $p$ & a prime number \\
  $\nu_p(n)$ &  the $p$-adic valuation of $n$ \\
  $n_p$ &  the $p$-part of the natural number $n$ ($n_p=p^{\nu_p(n)}$) \\
  $\Z_p$ &  the ring of $p$-adic integers \\
  $q$ &  a prime power \\
  $\F_q$ &  the finite field with $q$ elements \\
  $s$ &  the characteristic of $\F_q$ \\
  $\rchi_r(\operatorname{GL}^\pm_n(\F_q))$ & equivariant \Euc\ $ \rchi_r(\Li_n^\pm(\F_q)^*, \operatorname{GL}^\pm_n(\F_q))$ (Definition~\ref{defn:chir}) \cite[Definition~1.2]{jmm:eulergl+}   \\
  $\rchi_r(p,\operatorname{GL}^\pm_n(\F_q))$ & $p$-primary \Euc\ $ \rchi_r(p,\Li_n^\pm(\F_q)^*, \operatorname{GL}^\pm_n(\F_q))$ (Definition~\ref{defn:pprimeuc}) \cite[Definition~4.2]{jmm:eulergl+} \\
  $\binom m{-k}$ & the signed binomial coefficient $(-1)^k\binom mk$                                                 
\end{tabular}

\section{The general unitary group $\GU nq$}
\label{sec:building}
Let $q$ be a prime power, $n \geq 1$ a natural number, and
$V_n(\F_{q^2}) = \F_{q^2}^n$ the vector space of dimension $n$ over
the field $\F_{q^2}$ with $q^2$ elements. The non-degenerate
sesquilinear form 
\begin{equation}\label{eq:unitaryform}
   \langle u,v \rangle = c\sum_{1 \leq i \leq n} (-1)^{i+1} u_iv_{n+1-i}^q \qquad
   u,v \in V_n(\F_{q^2})
\end{equation}
is Hermitian ($\gen{au,v} = a \gen{u,v}$, $\gen{u,v}^q = \gen{v,u}$,
$a \in \F_{q^2}$, $u,v \in V_n(\F_{q^2})$) when the constant
$c \in \F_{q^2}$ satisfies $c^{q-1}=(-1)^{n+1}$. The {\em general unitary
group\/} $\GU nq$ \cite[\S 2.7]{GLSIII} is the group of all linear auto\m
s of $V_n(\F_{q^2})$ preserving the Hermitian bilinear form
\eqref{eq:unitaryform}.  Let $\varphi_q(g)$ denote the matrix obtained
from $g \in \GL n{q^2}$ by raising all entries to the power $q$. Then
$g$ lies in $\GU nq$ if and only if $gA(\varphi_q(g))^t=A$ where $A$
is the matrix whose only nonzero entries are a string of alternating
$+1$'s and $-1$'s running diagonally from upper right to lower left
corner.  The order of $\GU nq$ is \cite[(2.6.1)]{wall63}
\cite[(3.25)]{wilson2009}
  \begin{equation*}
  |\GU nq|= (q+1) |\SU nq| = q^{\binom{n}{2}}\prod_{1 \leq
    i \leq n} (q^i-(-1)^i) = \prod_{0 \leq i \leq
    n-1}(q^n-(-1)^{n-i}q^i)  
\end{equation*}
and there is a short exact sequence
\begin{equation}\label{eq:ses}
  \begin{tikzpicture}
  \node (dia) [matrix of math nodes, column sep={1cm}]  {
    1 & \SU nq & \GU nq & C_{q+1} & 1\\};
  \draw[->]  (dia-1-3) --node [above] {$\det$} (dia-1-4);
  \draw[->] (dia-1-1)--(dia-1-2);
  \draw[->] (dia-1-2)--(dia-1-3);
  \draw[->] (dia-1-4)--(dia-1-5.189);
    \end{tikzpicture}
  \end{equation}
  where $C_{q+1}$ is the order $q+1$ subgroup of the cyclic unit group
  $\F_{q^2}^\times$. We have $|\operatorname{GL}^-_n|(q) =(-1)^n |\operatorname{GL}^+_n|(-q)$ where
   $|\operatorname{GL}_n^+|(q) = \prod_{0 \leq i \leq n-1}(q^n-q^i)$ and
   $|\operatorname{GL}^-_n|(q) = \prod_{0 \leq i \leq n-1}(q^n-(-1)^{n-i}q^i)$ are the order \pol s
   \cite[p 207]{malle-testerman2011}
  for the general linear and unitary groups.
  The {\em special unitary group\/}
$\SU nq$ is generated by root group elements $x_{\widehat{\alpha}}(t)$
or $x_{\widehat{\alpha}}(t,u)$ of type I, II, and (for odd $n$) IV
\cite[Table 2.4]{GLSIII} and the general unitary group $\GU nq$ by
root groups together with the diagonal matrices
$\diag(z,1,\ldots,1,z^{-q})$ for $z \in \F_{q^2}^\times$.


A subspace of $V_n(\F_{q^2})$ is {\em totally isotropic\/} if the
Hermitian sesquilinear form \eqref{eq:unitaryform} vanishes completely
on it. Let $\Li^-_n(\F_q)$ be the poset of totally isotropic subspaces
in $V_n(\F_{q^2})$ and $\Li_n^-(\F_q)^*$ the subposet of {\em
  nontrivial\/} totally isotropic subspaces. The standard action of $\GL n{{q^2}}$ on subspaces of $V_n(\F_{q^2})$ restricts to an action of $\GU nq$ on $\Li_n^-(\F_q)^*$. The classifying simplicial complex of $\Li_n^-(\F_q)^*$, the flag complex of totally isotropic subspaces, is the building for $\GU nq$ \cite[\S 6.8]{abramenko_brown2008}.  We may replace the flag complex ${\Li}_n^-(\F_q)^*$ by the Brown subgroup poset $\cat S{\GU nq}{s+*}$ of nontrivial $s$-subgroups of $\GU nq$ where $s$ is the defining characteristic \cite[Theorem 3.1]{quillen78}.

The (non-equivariant) reduced \Euc s of the spherical posets $\Li_n^\pm(\F_q)^*$ are given by  
\begin{equation*}
    -\rchi(\Li_n^+(\F_q)^*) = (-1)^{n-1} q^{\binom n2}, \qquad
  -\rchi(\Li_n^-(\F_q)^*) = (-1)^{\lfloor n/2 \rfloor} q^{\binom n2}
\end{equation*}
according to the Solomon--Tits theorem \cite[Proposition~8.3]{curtis-lehrer-tits80} (or \cite[Example~3.10.2]{stanley97} for the case of
$\Li_n^+(\F_q)^*$).
    

\section{Self-dual polynomials over $\F_{q^2}$}
\label{sec:self-dual-polyn}

In the next lemma, we consider field extensions
$\F_q \subseteq \F_{q^{m_1}} \subseteq \F_{q^{m_2}}$ where
$1 \leq m_1 \leq m_2$. Let $\sigma_0, \sigma_1, \ldots, \sigma_n$ be
the elementary symmetric \pol s in $n \geq 1$ variables \cite[Example
1.74]{lidlnieder97} (where $\sigma_0$ stands for the constant \pol\
$1$).

\begin{lemma}\label{lemma:symmetric}
 Let $a_1,\ldots,a_n$ be $n$ elements of the field $\F_{q^{m_2}}$. Then
 \begin{equation*}
   \forall i \in \{0,1,\ldots,n\} \colon \sigma_i(a_1,\ldots,a_n) \in
   \F_{q^{m_1}} \iff
   \forall i \in \{0,1,\ldots,n\} \colon \sigma_i(a_1^{-q},\ldots,a_n^{-q}) \in
   \F_{q^{m_1}}
 \end{equation*}
\end{lemma}
\begin{proof}
  The $n$th elementary symmetric function is $\sigma_n(a_1,\ldots,a_n)
  = a_1 \cdots a_n$. Observe that
  \begin{equation*}
    \forall i \in \{0,1,\ldots,n\} \colon
    \sigma_i(a_1^{-1},\ldots,a_n^{-1})
  \sigma_n(a_1,\ldots,a_n) = \sigma_{n-i}(a_1,\ldots,a_n)
  \end{equation*}
  If all values of $\sigma_i(a_1,\ldots,a_n)$ are in the subfield
  $\F_{q^{m_1}}$, also all values of
  $\sigma_i(a_1^{-1},\ldots,a_n^{-1})$ and
  $\sigma_i(a_1^{-q},\ldots,a_n^{-q})
  =\sigma_i(a_1^{-1},\ldots,a_n^{-1})^q$ are in this subfield.
\end{proof}

\begin{defn}[Dual \pol ] \cite[Notation p.\ 13]{wall63}\label{defn:action}
  Let
  $p(x) = a_0x^m + a_1x^{m-1} + \cdots + a_{m-1}x + a_m \in
  \F_{q^2}[x]$ be a \pol\ of degree $m \geq 1$ with nonzero constant
  term (so that $a_0 \neq 0$ and $a_m \neq 0$). The dual \pol\ to $p(x)$ is
  \begin{equation*}
    \overline p(x) = a_0 \prod_{1 \leq i \leq m} (x-\alpha_i^{-q})
  \end{equation*}
  where $p(x) = a_0 \prod_{1 \leq i \leq m} (x-\alpha_i)$ with
  $\alpha_1,\ldots,\alpha_m$ in the splitting field for $p(x)$ over
  $\F_{q^2}$.
  If $p(x) = \overline p(x)$ we say that $p(x)$ is self-dual.
\end{defn}


We note that
\begin{itemize}
\item  dualization is involutory: $\overline{\overline p}
  = p$
\item dualization respects products:
  $\overline{p_1p_2} = \overline p_1 \overline p_2$
 \item  dualization respects divisibility: $p_1 \mid p_2 \iff \overline p_1
\mid \overline p_2$
\item a \pol\ (with nonzero constant term) is irreducible if and only
  its dual \pol\ is irreducible
\item if $p = a_0 \prod r_i^{e_i}$ is the canonical factorisation of
  the \pol\ $p$ \cite[Theorem 1.59]{irelandrosen90} then
  $\overline p = a_0 \prod \overline r_i^{e_i}$ is the canonical
  factorisation of the dual \pol
\end{itemize}

Although the dual of a \pol\ over $\F_{q^2}$ is defined in terms of elements of an extension of $\F_{q^2}$, it is actually again a \pol\ over $\F_{q^2}$ as Lemma~\ref{lemma:symmetric} shows that the coefficients of $\overline{p}(x)$ lie in $\F_{q^2}$ if those of $p(x)$ do.

 \begin{prop}\label{prop:sdcrit}
   Let
   $p(x) = a_0x^m + a_1x^{m-1} + \cdots + a_{m-1}x + a_m \in
   \F_{q^2}[x]$ be a \pol\ as in Definition~\ref{defn:action} with
   $a_0 \neq 0$ and $a_m \neq 0$. The dual \pol\ $\overline p(x)$ is given
   by
   \begin{equation*}
     a_m^q \overline p(x) = a_0(a_m^qx^m + a_{m-1}^qx^{m-1}+\cdots+a_1^qx+a_0^q)
   \end{equation*}
   and $p(x)$ is self-dual if and only if its coefficients satisfy the
   equation
   \begin{equation*}
     a_m^q(a_0,a_1,\ldots,a_{m-1},a_m) = a_0(a_m^q,a_{m-1}^q,\ldots,a_1^q,a_0^q)
   \end{equation*}
 \end{prop}
 \begin{proof}
   The reciprocal \cite[Definition 3.12]{lidlnieder97}  to the \pol\ $p(x)$ is 
   \begin{equation*}
     p^*(x) = x^mp(x^{-1}) = a_mx^m + a_{m-1}x^{m-1} + \cdots + a_{1}x
     + a_0 = a_m \prod_{1 \leq i \leq m} (x-\alpha_i^{-1})
   \end{equation*}
    and thus
    \begin{equation*}
      a_m^q \prod_{1 \leq i \leq m} (x-\alpha_i^{-q}) =
      a_m^qx^m + a_{m-1}^qx^{m-1} + \cdots + a_{1^q}x+ a_0^q
    \end{equation*}
     since the Frobenius map $\sigma_q(x)=x^q$ is a field auto\m\ of $\F_{q^2}$.
     The dual \pol\ is
    \begin{equation*}
      \overline p(x) =  a_0 \prod_{1 \leq i \leq m} (x-\alpha_i^{-q})=
      \frac{a_0}{a_m^q}
      (a_m^qx^m + a_{m-1}^qx^{m-1} + \cdots + a_{1}^qx + a_0^q)
    \end{equation*}
    and hence
    \begin{multline*}
      p(x) = \overline p(x) \iff a_m^q p(x) = a_m^q \overline p(x) \\ \iff
      a_m^q(a_0x^m + a_1x^{m-1} + \cdots + a_{m-1}x + a_m) =
      a_0(a_m^qx^m + a_{m-1}^qx^{m-1} + \cdots + a_{1}^qx + a_0^q)
    \end{multline*}
    which is the criterion of the proposition.
 \end{proof}

 If $g$ is a unitary auto\m\ of a vector space over $\F_{q^2}$ and
$p(x)$ the \pol\ of Definition~\ref{defn:action} then
\begin{equation}
  \label{eq:conjugation}
  a_0^q\gen{p(g)x,y} = a_m\gen{g^m(x),\bar p(g)(y)}
\end{equation}
for all vectors $x,y$.

\begin{lemma}\label{lemma:charpolselfdual}
  The characteristic \pol\ $c_g$ of any unitary auto\m\ $g \in \GU nq$
  is self-dual.
\end{lemma}
\begin{proof}
  Let $r \in \F_{q^2}[x]$ be an irreducible \pol . Then
  \begin{equation*}
    r \nmid c_g \iff
    \text{$r(g)$ is invertible}
    \stackrel{\eqref{eq:conjugation}}{\iff}
    \text{$\bar r(g)$ is invertible} \iff
    \bar r \nmid c_g 
  \end{equation*}
  or, equivalently, $r \mid c_g \iff \bar r \mid c_g$. This shows that
  $c_g=\bar c_g$.
\end{proof}

\begin{cor}\cite[Proof of (ii), p.\ 35]{wall63} \label{cor:sdcount}
  The number of self-dual monic \pol s in $\F_{q^2}[x]$ of degree $m > 0$
  with nonzero constant term is $q^{m}+q^{m-1}$.
\end{cor}
 \begin{proof}
  A monic \pol\ of degree $m$,
  $p(x) = x^m + a_1x^{m-1} + \cdots + a_{m-1}x + a_m \in \F_{q^2}[x]$
  with $a_m \neq 0$ is by Proposition~\ref{prop:sdcrit} self-dual if
  and only if
 \begin{equation*}
   a_m^q(a_1,a_2,\ldots,a_{m-1},a_m)=(a_{m-1}^q,\ldots,a_1^q,1)
 \end{equation*}
 or, equivalently,
 \begin{equation}\label{eq:monicsdcrit}
   a_m^{q+1}=1 \text{\ and\ } (a_1,\ldots,a_{m-1}) = a_m(a_{m-1}^q,\ldots,a_1^q)
 \end{equation}
 Suppose first that $m=2k+1$ is odd. There are $q+1$ elements $a_m$ in
 $\F_{q^2}$ such that $a_m^{q+1}=1$. For $1 \leq j \leq k$, let $a_j$
 be any element of $\F_{q^2}$ and put $a_{m-j}=a_ma_j^q$. Then
 $a_ma_{m-j}^q=a_j^{q^2}=a_j$. This shows that the self-duality
 criterion \eqref{eq:monicsdcrit} has $(q+1)q^{2k} = q^m+q^{m-1}$
 solutions. Suppose next that $m=2k$ is even. The coefficient $a_m$
 can again be chosen in exactly $q+1$ ways. For each $j$ with
 $1 \leq j \leq k-1$, the coefficient $a_j$ can be chosen freely in
 $\F_{q^2}$ and we let $a_{m-j}=a_ma_j^q$. There are $q=(q-1)+1$
 possibilities for choosing the coefficient $a_k$ such that
 $a_k=a_ma_k^q$ as $a_m^{q+1}=1$. Thus the self-duality criterion
 \eqref{eq:monicsdcrit} has $(q+1)q^{2k-2}q = q^m + q^{m-1}$
 solutions.
\end{proof}

\begin{defn}[See Figure~\ref{fig:Apm}]\label{defn:Adq2}
  For every integer $d \geq 1$,
  \begin{itemize}
  \item $\IM dq$ is the number of Irreducible Monic \pol s of degree
    $d$ over $\F_{q}$ with nonzero constant term
\item $\SDIM -dq$  is the number of Self-Dual Irreducible Monic \pol s of
  degree $d$ over $\F_{q^2}$ with nonzero constant term 
\item $\SDIM +dq =\frac{1}{2}(\IM d{q^2}- \SDIM -dq)$ is the number of unordered
  pairs of non-self-dual irreducible monic \pol s of degree $d$ over
  $\F_{q^2}$ with nonzero constant term
  \end{itemize}
\end{defn}

    \begin{figure}[t]  
    \centering
   \begin{tabular}[h]{>{$}c<{$}|*{6}{>{$}c<{$}}}
     {} & d=1 & d=2  & d=3 & d=4 & d=5 & d=6  \\ \hline
     d\IM dq & q-1 & q^2-q & q^3 - q & q^4 - q^2 &
    q^5 - q &
    q^6 - q^3 - q^2 + q \\
    d\IM d{q^2}  & q^2 - 1 & q^4 - q^2 &
    q^6 - q^2 &
    q^8 - q^4 &
    q^{10} - q^2 & q^{12} - q^6 - q^4 + q^2 \\
    d\SDIM -dq  &   q+1  &   0  &   q^3-q  &   0  &   q^5-q  &    0\\
    2d\SDIM +dq  &  q^2-q-2  &   q^4-q^2  &   q^6 - q^3 - q^2 + q &
    q^8 - q^4 &
    q^{10} - q^5 - q^2 + q &
    q^{12} - q^6 - q^4 + q^2 
                                             \end{tabular}
    \caption{The \pol s $\IM dq$, $\IM d{q^2}$ and $\SDIM {\pm}dq$ for
    $d=1,\ldots,6$} 
    \label{fig:Apm}
  \end{figure}

  For all $d \geq 1$, $\IM d{q^2} = \sum_{d \mid n} \mu(n/d)(q^{2d}-1)$
  \cite[Corollary~3.21]{lidlnieder97}
  (simplifying to $\IM d{q^2} = \sum_{d \mid n} \mu(n/d)q^{2d}$ when
  $d>1$). The well-known identities \cite[p.\ 258]{wall99} 
\begin{equation}\label{eq:Adq2}
  \prod_{d \geq 1} \frac{1}{(1-x^d)^{\IM dq}} = \frac{1-x}{1-qx}, \qquad
  \prod_{d \geq 1} \frac{1}{(1-x^d)^{\IM d{q^2}}} = \frac{1-x}{1-q^2x}
\end{equation}
are easily proved using \cite[Lemma~3.7]{jmm:eulergl+}.
When $d=1$, $\IM 1{q^2} = q^2-1$ (represented by the \pol s $x-\lambda$,
$\lambda \in \F_{q^2}^\times$), $\SDIM -1q=q+1$ (represented by the
\pol s $x-\lambda$, $\lambda \in \F_{q^2}^\times$,
$\lambda=\lambda^{-q}$) and
$\SDIM +1q = \frac{1}{2}(q^2-1-(q+1)) = \frac{1}{2}(q+1)(q-2)$
(represented by the pairs $(x-\lambda,x-\lambda^{-q})$,
$\lambda \in \F^\times_{q^2}$, $\lambda \neq \lambda^{-q}$).

The next proposition shows among other things that  self-dual
irreducible 
\pol s have odd degrees,
ie that $\SDIM -dq=0$ for all even $d$.

\begin{prop}\label{prop:oddm}
  Let $p(x) \in \F_{q^2}[x]$ be a self-dual irreducible monic \pol\ of
  degree $m \geq 1$ over $\F_{q^2}$ with $p(0) \neq 0$. Then $m$ is
  odd and
  \begin{equation*}
    p(x) = \prod_{0 \leq j \leq m-1} (x-\lambda^{q^{2j}})
  \end{equation*}
  where $\lambda \in \F_{q^{2m}}$, $\lambda^{q^m+1}=1$, and all the
  elements $\lambda,\lambda^{q^2},\ldots,\lambda^{q^{2m-2}}$ are
  distinct.
\end{prop}
\begin{proof}
  Let $p(x)$ be a monic irreducible \pol\ $p(x)$ of degree $m$ over
  $\F_{q^2}$. The field $\F_{q^{2m}}$ contains an element $\lambda$
    such that
    \begin{equation*}
    p(x) = \prod_{0 \leq j \leq m-1} (x-\lambda^{q^{2j}})
  \end{equation*}
  and all the elements
  $\lambda,\lambda^{q^2},\ldots,\lambda^{q^{2m-2}}$ are distinct
  \cite[Theorem 2.14]{lidlnieder97}. By self-duality
  $\lambda^{-q} = \lambda^{q^{2k}}$ for a unique integer $k$ with
  $0 \leq k \leq m-1$.

  Assume first that $p(x)$ has degree $m=2$. The roots of $p(x)$ are
  $\{\lambda,\lambda^{q^2}\}$ where $\lambda^{-q}$ equals $\lambda$ or
  $\lambda^{q^2}$ by self-duality. In the first case,
  $1=\lambda \lambda^{-1} = \lambda \lambda^q = \lambda^{q+1}$ and
  $\lambda^{q^2-1} = (\lambda^{q+1})^{q-1} = 1$. In the second case,
  $\lambda^q = \lambda^{-q^2} = (\lambda^{-q})^q = (\lambda^{q^2})^q =
  \lambda^{q^3}$ and $1 = \lambda^{q^3-q} = (\lambda^{q^2-1})^q$ so
  $\lambda^{q^2-1}=1$ also here.  In both cases, we have that
  $\lambda, \lambda^{q^2} \in \F_{q^2}^\times$. Since this contradicts
  irreducibility of $p(x)$ over $\F_{q^2}$, monic irreducible
  self-dual \pol s of degree $2$ do not exist.

  Assume next that $m>2$. Since
  $\lambda^{q^2} = (\lambda^{-q})^{-q} = \lambda^{q^{4k}}$ it follows
  that $m$ divides $2k-1$ and is odd.  Furthermore, $k$ equals $1$ or
  $\frac{1}{2}(m+1)$ as $k$ is at most $m-1$. However, $k=1$ implies
  $\lambda^{q^2} = \lambda^{q^{4k}} = \lambda^{q^4}$ contradicting
  that $\lambda^{q^2}$ and $\lambda^{q^4}$ are distinct when
  $m \geq 3$. From $2k=m+1$ we get
  $\lambda^{-q} =\lambda^{q^{2k}} = \lambda^{q^{m+1}}$, equivalently,
  $\lambda^{-1} = \lambda^{q^m}$ or $\lambda^{q^m+1}=1$.
\end{proof}

The next count of self-dual irreducible monic \pol s in $\F_{q^2}[x]$  is closely related to the classical count of irreducible monic \pol s or self-reciprocal irreducible monic \pol s in $\F_q[x]$ \cite[Corollary~3.21, Theorem~3.25]{lidlnieder97} \cite[Theorem~3]{meyn1990}.

\begin{lemma} \label{lemma:meyn}
  Let $m \geq 1$ be an odd integer.
  \begin{enumerate}
  \item  The self-dual irreducible monic \pol s in $\F_{q^2}[x]$ with nonzero constant term whose degrees
  divide the odd integer $m \geq 1$ are precisely the irreducible factors of the
  \pol\ $x^{q^m+1}-1 \in \F_{q^2}[x]$. \label{lemma:meyna}
\item $\sum_{d \mid m} d \SDIM -dq = q^m+1$ and
  $m \SDIM -mq = \sum_{d \mid m} \mu(m/d) (q^d+1)$ for any odd integer $m\geq 1$.
  \label{lemma:meynb}
  \end{enumerate}
\end{lemma}
\begin{proof}
  \eqref{lemma:meyna}
  Let $p(x)$ be an irreducible factor of $x^{q^m+1}-1$. Obviously, $p(0) \neq 0$. If $\alpha$ is a root of $p(x)$ in its splitting field
  then $\alpha^{q^m+1}=\alpha$. Therefore $\alpha^{-q} = \alpha^{q^m}$ is also a root of $p(x)$. This shows that $p(x)$ is self-dual (Definition~\ref{defn:action}).

  Next, let $p(x)$ be a self-dual irreducible monic \pol\ with nonzero constant term of degree $d$ dividing $m$.
  According to Proposition~\ref{prop:oddm}, $p(x)$ has a root $\lambda \in \F_{q^{2d}}$ such that $\lambda^{q^{d+1}}-1=0$. Then $p(x)$ divides $x^{q^{d+1}}-1$ which divides $x^{q^{m+1}}-1$ as $d \leq m$ \cite[Lemma~2.12, Corollary~3.7]{lidlnieder97}.

  \noindent \eqref{lemma:meynb} The \pol\ $x^{q^{2m}}-x= x(x^{q^{2m-1}}-1) \in \F_{q^2}[x]$ has no multiple roots according to the standard criterion of \cite[Theorem~1.68]{lidlnieder97}. The \pol\ $x^{q^{m+1}}-1$ is a factor of
  $x^{q^{2m}}-x$ by \cite[Corollary~3.7]{lidlnieder97} and hence also has no multiple roots. From \eqref{lemma:meyna} it now follows that $x^{q^{m+1}}-1$ is the product of all self-dual irreducible \pol s with nonzero constant terms of degrees dividing $m$.  The second assertion is the M{\"o}bius inversion of the first one which is a count of degrees.
 \end{proof}

  \begin{cor}\label{cor:Apmq}
    The arithmetic functions $\IM nq$ and $\SDIM {\pm}nq$ of Definition~\ref{defn:Adq2} satisfy the relations
    \begin{equation*}
    \SDIM -nq =
    \begin{cases}
      q+1 & n=1\\
      \IM nq & \text{$n>1$ odd} \\
      0 & \text{$n>0$ even} 
    \end{cases}
    \qquad
    \SDIM +nq =
    \begin{cases}
      \frac{1}{2}q(q-1)-1 & n=1 \\
      \IM{2n}q & n > 1 
    \end{cases}
  \end{equation*}
  \end{cor}
  \begin{proof}
  For $n=1$, the $\SDIM -1q =q+1$ self-dual irreducible monic \pol s
  are the \pol s $x-\lambda$ with $\lambda \in \F_{q^2}$ such that
  $\lambda^{q+1}=1$. For odd $n>1$, Lemma~\ref{lemma:meyn}.\eqref{lemma:meynb} shows that 
  $\SDIM -nq = \frac{1}{n} \sum_{d \mid n} \mu(n/d)q^d = \IM nq$, the
  number of irreducible \pol s of degree $n$ over $\F_q$ \cite[Chapter
  2, Corollary]{rosen2002} \cite[Theorem 3.25]{irelandrosen90}.
  When $n>1$ is odd
  \begin{multline*}
    \IM{2n}q = \frac{1}{2n} \sum_{ D \mid 2n} \mu(2n/D)q^D =
    \frac{1}{2n} \sum_{ d \mid n} \mu(n/d)q^{2d} +
    \frac{1}{2n} \sum_{d \mid n} \mu(2n/d)q^d \\ =
    \frac{1}{2n} \sum_{ d \mid n} \mu(n/d)q^{2d} -
    \frac{1}{2n} \sum_{d \mid n} \mu(n/d)q^d =
    \frac{1}{2}(\IM n{q^2} - \IM nq) = \SDIM +nq
  \end{multline*}
  where we use that $\mu(2k)=-\mu(k)$ for odd $k \geq 1$.
  When $n>0$ is even
  \begin{multline*}
    \IM{2n}q  = \frac{1}{2n} \sum_{ D \mid 2n} \mu(2n/D)q^D =
    \frac{1}{2n} \sum_{ d \mid n} \mu(n/d)q^{2d} +
    \frac{1}{2n} \sum_{\substack{d \mid n \\ \text{$d$ odd}}} \mu(2n/d)q^d \\ =
    \frac{1}{2n} \sum_{ d \mid n} \mu(n/d)q^{2d} =
    \frac{1}{2}\IM n{q^2}  = \SDIM +nq 
  \end{multline*}
  where we use that an even divisor of $2n$ has the form $2d$ for a divisor $d$ of $n$,
  an odd divisor of $2n$ is a divisor of $n$, and $\mu(2k)=0$ even $k \geq 2$.
\end{proof}

\begin{cor}\label{cor:sumSDIM}
  $\sum_{d \mid n} d\SDIM -dq = q^{n/n_2} +1$ and
  $\sum_{d \mid n} d\SDIM +dq = \frac{1}{2}(q^{2n} - q^{n/n_2})-1$ for any natural number $n \geq 1$. 
\end{cor}
\begin{proof}
  To get the first equation,
  \begin{equation*}
    \sum_{d \mid n} d\SDIM -dq = 2 + \sum_{d \mid n/n_2} d\IM dq = 2+q^{n/n_2}-1 = q^{n/n_2}+1  
  \end{equation*}
  we use Corollary~\ref{cor:Apmq} and \cite[Corollary~3.21]{lidlnieder97}. The second equation,
  \begin{equation*}
    \sum_{d \mid n} d\SDIM +dq =
    \frac{1}{2}(\sum_{d \mid n} d\IM d{q^2} - \sum_{d \mid n} d\SDIM -dq) =
    \frac{1}{2}(q^{2n}-1-q^{n/n_2}-1) = \frac{1}{2}(q^{2n} - q^{n/n_2})-1
  \end{equation*}
  follows because $\SDIM +dq = \frac{1}{2}(\IM d{q^2} - \SDIM -dq)$ (Definition~\ref{defn:Adq2}).
\end{proof}

\section{Equivariant reduced \Euc s of products}
\label{sec:equiv-reduc-euc}


This short section establishes a multiplicative property of
equivariant \Euc s for use in the proof of the crucial
Lemma~\ref{lemma:recur}.

\begin{lemma}\label{lemma:Eucjoin}
  $-\rchi(P_1 \ast \cdots \ast P_t) = \prod_{1 \leq i \leq t}
  -\rchi(P_i)$ for finitely many finite posets $P_1,\ldots,P_t$.
\end{lemma}
\begin{proof}
  The join $P \ast Q$, of the finite posets $P$ and $Q$, is the poset
  $P \coprod Q$ where all elements of $P$ are $<$ all elements of
  $Q$. The $n$-simplices of the join are $n$-simplices of $P$, $i$
  simplices of $P$ joined to $j$-simplices of $Q$ where $i+j=n-1$, and
  $n$-simplices of $Q$.  Alternatively, when we regard a poset as
  having a single cell $\emptyset$ in degree $-1$, the $n$-simplices
  of the join are all $i$-simplices of $P$ joined to all $j$-simplices
  of $Q$ where $i+j=n-1$. In other words
  $c_n(P \ast Q) = \sum_{i+j=n-1} c_i(P)c_j(Q)$, where $c_n$ stands for
  the number of $n$-simplices. The reduced \Euc\ of the join is
  \begin{multline*}
    -\rchi(P \ast Q) = \sum_{n \geq -1} (-1)^{n-1}c_n(P \ast Q)
                           = \sum_{n \geq -1} \sum_{i+j=n-1}
                           (-1)^ic_i(P)(-1)^jc_j(Q)
                           = \sum_{i \geq -1} c_i(P) \sum_{j \geq -1}
                           c_j(Q) \\
                           = \rchi(P) \rchi(Q)
                           = (-\rchi(P)) (-\rchi(Q))
  \end{multline*}
  Proceeding by induction we get the formula for the reduced \Euc\ of
  finite joins of finite posets.
\end{proof}

For a finite poset $P$ with a least element $\widehat 0$, let
$P^* = P - \{ \widehat 0 \}$ be the induced subposet obtained by removing
$\widehat 0$ from $P$.  Let $G_i$ be finite groups and $P_i$ finite
$G_i$-posets with least elements indexed by the finite set $I$. The
product poset $\prod_{i \in I}P_i$ is a finite
$\prod_{i \in I} G_i$-poset with a least element.

\begin{lemma}\label{lemma:chiproduct}
  The classical and the equivariant \Euc s of the $\prod_{i \in I} G_i$-poset 
  $\big(\prod_{i \in I} P_i \big)^*$ are given by
  \begin{equation*}
    -\rchi\Big( \big(\prod_{i \in I} P_i \big)^*\Big) = \prod_{i \in
      I} -\rchi(P_i^*), \qquad
    -\rchi_r\Big( \big(\prod_{i \in I} P_i \big)^*, \prod_{i \in I}
    G_i \Big) = \prod_{i \in I} -\rchi_r(P_i^*,G_i) 
  \end{equation*}
  where $r \geq 1$.
\end{lemma}
\begin{proof}
  If $P_1$ and $P_2$ are finite posets with least elements then
  Lemma~\ref{lemma:Eucjoin} implies 
  $-\rchi((P_1 \times P_2)^*) = (-\rchi(P_1^*)) (-\rchi(P_2^*))$
  because
  $(P_1 \times P_2)^* = (P_1 \times P_2)_{>(\widehat{0},\widehat{0})}
  =(P_1)_{>0} \ast (P_2)_{>0} =P_1^* \ast P_2^*$ by \cite[Proposition
  1.9]{quillen78} \cite[Theorem 5.1.(c)]{walker81}.  The general
  formula for the classical \Euc\ follows by induction over the
  cardinality of the index set $I$. Proceed exactly as in \cite[Lemma
  2.3]{jmm:eulergl+} to obtain the formula for the equivariant \Euc s
\end{proof}

\section{Semisimple classes of the general  unitary group}
\label{sec:semi-simple-unitary}

Conjugacy classes in the general linear group $\gl$ or the general
unitary group $\gu$ are classified by functions from the set of
irreducible \pol s in $\F_q[x]$ or $\F_{q^2}[x]$ to the set of
partitions \cite{carter81} \cite[\S2.1, \S2.2]{burkett_nguyen2013}
\cite[Proposition 1A]{fong_srinivasan82}.

An element of $\gu$ is {\em semisimple\/} if it is diagonalisable over
the algebraic closure of $\F_q$
\cite[\S1.4]{carter:finiteLie}. Alternatively, the semisimple elements
of $\GU nq$ are precisely the $q$-regular elements (the elements of
order prime to $q$); see \cite[\S2]{thevenaz92poly}.  A {\em
  semisimple\/} or {\em $q$-regular\/} class in $\gu$ is the conjugacy
class of a semisimple (= $q$-regular) element.

\begin{cor}\label{cor:semisimplecharpol}
  $\operatorname{GL}^\pm_n(\F_q)$
  contains exactly $q^n \mp q^{n-1}$ semisimple classes for any
  $n \geq 1$.
  Two semisimple elements of $\operatorname{GL}^{\pm}_ n(\F_q)$ are
  conjugate if and only if their
  characteristic \pol s are identical. 
\end{cor}
\begin{proof}
  The number of semisimple classes is given by a general result of Steinberg \cite[Theorem~3.7.6]{carter:finiteLie}.
  The second statement is an immediate consequence of the
  classification of $q$-regular classes in $\gu$ mentioned above.
\end{proof}

For a $G$-poset $\Pi$, let $\sim$ be the equivalence relation between $G$-poset endo\m s of $\Pi$ generated by the relation $f_0 \sim f_1$ if $f_0(x) \leq f_1(x)$ for all $x \in \Pi$. We say that $\Pi$ is {\em $G$-poset contractible\/} if there is a $G$-fixed point $x_0$ in $\Pi$ such that $1_\Pi \sim x_0$ where $1_\Pi$ is the identity map of $\Pi$ and $x_0$ is the contant map with value $x_0$ \cite[\S2]{jmm:eulergl+}. If $\Pi$ is $G$-poset contractible then any subposet $C_{\Pi}(X)$ fixed by a subset $X$ of $G$ is poset contractible.

\begin{lemma}\label{lemma:contractability}
  For $n>1$, the poset $C_{\Li_n^-(\F_q)^*}(g)$ is
  $C_{\gu}(g)$-poset contractible unless $g \in \gu$ is semisimple.
\end{lemma}
\begin{proof}
  This is proved in \cite[\S4]{webb87} once we recall Quillen's
  identification \cite{quillen78} of $\Li_n^-(\F_q)^*$ with the Brown
  poset of nontrivial $s$-subgroups of $\gu$ where $s$ is the
  characteristic of $\F_q$.
\end{proof}

The next lemma
facilitates a recursive approach to the equivariant \Euc s
$\rchi_r(\gu)$.
The characteristic \pol\ of any unitary auto\m\ is self-dual by 
Lemma~\ref{lemma:charpolselfdual} 
and thus admits an essentially unique factorisation of the form
$\prod r_i^{m_i^-} \times \prod_j (s_j \bar s_j)^{m_j^+}$ where the
$r_i$ are distinct self-dual irreducible monic \pol s and the $s_j$
are distinct non-self-dual irreducible monic \pol s.
($[\gu]$ denotes the set of conjugacy classes in $\gu$.)

\begin{lemma}\label{lemma:recur}
  For $n > 1$ and $r \geq 1$, the $(r+1)$th equivariant \Euc\ of the $\gu$-poset
  $\Li^-_n(\F_q)^*$ is 
  \begin{equation*}
    \rchi_{r+1}(\GU nq) = \sum_{\substack{[g] \in [\GU nq] \\ \GCD {q,|g|}=1}}
    \rchi_{r}(C_{\Li_n^-(\F_q)^*}(g), C_{\gu}(g))
  \end{equation*}
 where the contribution from the semisimple class $g$ with characteristic
  \pol\ $\prod r_i^{m_i^-} \times \prod_j (s_j \bar s_j)^{m_j^+}$ is given by
  \begin{equation*}
    -\rchi_{r}(C_{\Li_n^-(\F_q)^*}(g), C_{\gu}(g)) =
    \prod_i -\rchi_r(\GU{m_i^-}{q^{d_i^-}}) \times \prod_j +\rchi_r(\GL{m_j^+}{q^{2d_j^+}})
  \end{equation*}
  for $\deg r_i = d_i^-$, $\deg s_j = d_j^+$ and
  $\sum_i m_i^-d_i^- + \sum_j 2m_j^+d_j^+ =n$. 
\end{lemma}
\begin{proof}
  View the $n$-dimensional unitary geometry $V$ as an
  $\F_{q^2}[x]$-module via the action of $g$.  Since $g$ is semisimple
  the $\F_{q^2}[x]$-module $V$ is
   \begin{equation*}
     V = \bigoplus_{r_i = \bar r_i} \ker(r_i(g)) \oplus
     \bigoplus_{s_j \neq \bar s_j} \ker(s_j(g)) \oplus \ker(\bar
     s_j(g)) 
     = \bigoplus_{r_i = \bar r_i} (\F_{q^2}[x]/(r_i(x)))^{m_i^-}  \oplus \bigoplus_{s_j \neq
       \bar s_j} (\F_{q^2}[x]/(s_j(x))  \oplus \F_{q^2}[x]/(\bar s_j(x)))^{m_j^+} 
   \end{equation*}
   The direct summands, $\ker r_i(g)$ and
   $\ker(s_j(g)) \oplus \ker(\bar s_j(g))$, in this decomposition of
   $V$ are pairwise orthogonal. For example, let $r_{i_1}$ and
   $r_{i_2}$ be two distinct self-dual irreducible factors of the
   characteristic \pol . For $v_1 \in \ker(r_{i_1}(g))$ and
   $v_2 \in \ker(r_{i_2}(g))$, the inner products
   $\gen{r_{i_2}^{m_{i_2}^-}(g)v_1,v_2}$ and
   $\gen{g^{d_{i_2}^-m_{i_2}^-}v_1, r_{i_2}^{m_{i_2}^-}(g)v_2}=0$
   agree up to a nonzero scalar by \eqref{eq:conjugation}. Since
   $r_{i_2}^{m_{i_2}^-}(g)$ defines an auto\m\ of $\ker(r_{i_1}(g))$,
   this shows that $\ker(r_{i_1}(g)) \perp
   \ker(r_{i_2}(g))$. Similarly,
   $\ker(s_i(g)) \perp (\ker(r_j(g)) \oplus \ker(\bar r_j(g)))$ and
   $\ker(r_{j_1}(g)) \perp (\ker(r_{j_2}(g)) \oplus \ker(\bar
   r_{j_2}(g)))$ for distinct factors $r_{j_1}$ and $r_{j_2}$. Thus
   all summands $\ker r_i(g)$ and
   $\ker(s_j(g)) \oplus \ker(\bar s_j(g))$ are non-degenerate unitary
   geometries.
   
    The centraliser of $g$ in
the general unitary group of $V$ is the group \cite{wall63} \cite[Proposition
1A]{fong_srinivasan82} \cite[Lemma 2.3]{burkett_nguyen2013}
\cite[Lemma 3.3]{tiep_zaless2004}
\begin{equation*}
     C_{\operatorname{GL}^-(V)}(g) =
     \prod_i \GU {m_i^-}{q^{d_i^-}} \times \prod_j
      \GL {m_j^+}{q^{2d_j^+}} 
    \end{equation*}
    of unitary $\F_{q^2}[x]$-auto\m s and the centraliser of $g$ in
    the poset of totally isotropic subspaces of $V$ is the poset
    \begin{equation*}
            C_{L^-(V)}(g) = \prod_i L^-(\ker r_i(g)) \times \prod_j L^-(
            \ker(s_j(g)) \oplus \ker(\bar s_j(g))) 
    \end{equation*}
    of totally isotropic $\F_{q^2}[x]$-subspaces.  The representation
    of $\GU {m_i^-}{q^{d_i^-}}$ in
    $\ker(s_i(g)) \cong \F_{q^{2d_i^-}}^{m_i^-}$ is standard. We now
    turn to the representation of $\GL {m_j^+}{q^{2d_j^+}}$ in
    $\ker(s_j(g)) \oplus \ker(\bar s_j(g)) \cong (\F_{q^{2d_j^+}}
    \oplus \F_{q^{2d_j^+}})^{m_j^+} =\F_{q^{2d_j^+}}^{2m_j^+}$
    described in \cite[\S 1, p 112, 1)]{fong_srinivasan82}.

    The Kleidman--Liebeck Theorem \cite{kleidman_liebeck90}
    \cite[Theorem 3.9]{wilson2009} lists certain natural subgroups of
    the general unitary groups. The unitary $2m$-geometry
    $V_{2m}(\F_{q^2})$ over $\F_{q^2}$ has a basis
    $e_1,\ldots,e_m,f_1,\ldots,f_m$ such that $\gen{e_i,f_i}=1$,
    $1 \leq i \leq m$, are the only nonzero Hermitian inner products
    between the basis vectors \cite[Proposition
    2.3.2]{kleidman_liebeck90}. Write
    $V_{2m}(\F_{q^2})=V_1 \oplus V_2$ as the direct sum of the two
    maximal totally isotropic subspaces $V_1$ and $V_2$ spanned by
    $e_1,\ldots,e_m$, and $f_1,\ldots,f_m$, respectively.  The
    representation of $\GL m{q^2}$ in $\GU {2m}q$ given by
  \begin{equation*}
    \GL m{q^2} \ni A \to
    \begin{pmatrix}
      A & 0 \\ 0 & A^{-1\alpha t}
    \end{pmatrix} \in \GU {2m}q
  \end{equation*}
  stabilises the direct sum
  decomposition $V = V_1 \oplus V_2$ \cite[Lemma 4.1.9, Table 4.2.A,
  Lemma 4.2.3]{kleidman_liebeck90}.  (The matrix $A^{-1\alpha t}$ is
  the conjugate-transpose of the inverse of $A$ so that
  $\gen{Av_1, A^{-1\alpha t}v_2} = \gen{A^{-1}Av_1,v_2} =
  \gen{v_1,v_2}$ for $v_1 \in V_1$, $v_2 \in V_2$.) The stabiliser of
  $g$ in the poset of totally isotropic subspaces of
  $V_{2m}(\F_{q^2})$ is the $\GL mq$-poset of pairs of orthogonal subspaces
  \begin{equation*}
    \Sigma\Li_m^+(\F_{q^2}) = \{ (U_1,U_2) \mid U_1 \leq V_1, U_2 \leq
    V_2, U_1 \perp U_2 \}
  \end{equation*}
  The subposet $\Sigma\Li_m^+(\F_{q^2})^*$, obtained from $\Sigma\Li_m^+(\F_{q^2})$ by
  removing the pair $(0,0)$, is $\GL mq$-homotopy equivalent to the
  suspension \cite[\S 3]{walker88} of $\Li_m^+(\F_{q^2})^*$: Let $\{1,2\}$ be the discrete
  poset of two incomparable points. The two $\GL mq$-poset \m s
  \begin{center}
    \begin{tikzpicture}[>=stealth', baseline=(current bounding box.-2)]
  \matrix (dia) [matrix of math nodes, column sep=25pt, row sep=20pt]{
   \{1,2\} \ast \Li_m^+(\F_{q^2})^* &  \Sigma\Li_m^+(\F_{q^2})^*\\};
  \draw[->] ($(dia-1-1.east)+(0,0.1)$) -- ($(dia-1-2.west)+(0,0.1)$) 
  node[pos=.5,above] {$f$}; 
  \draw[->] ($(dia-1-2.west)-(0,0.1)$) -- ($(dia-1-1.east)-(0,0.1)$)
   node[pos=.5,below] {$g$} ; 
 \end{tikzpicture}
\end{center}
given by $f(1,U)=(U,0)$, $f(2,U)=(0,U)$, and
\begin{equation*}
  g(U_1,U_2) =
  \begin{cases}
    (1,U_1) & U_1 \neq 0 \\ (2,U_2) & U_1 = 0
  \end{cases} \end{equation*} are homotopy equivalences as $gf$ is the
identity of the suspension of $\Li_m^+(\F_{q^2})^*$ and $fg$ is
homotopic to the identity of $\Sigma\Li_m^+(\F_{q^2})^*$ as
$fg(U_1,U_2) \leq (U_1,U_2)$.
By the product formula in Lemma~\ref{lemma:Eucjoin},
\begin{equation*}
  -\rchi_r(\Sigma\Li_m^+(\F_{q^2})^*,\GL m{q^2}) =
  - \rchi_r(\{1,2\} \ast \Li_m^+(\F_{q^2})^*,\GL m{q^2}) =
   \rchi_r(\Li_m^+(\F_{q^2})^*,\GL m{q^2})  
\end{equation*}
and the formula of the lemma is a consequence of the product formula
in Lemma~\ref{lemma:chiproduct}.
\end{proof}

Observe that the contribution of a $q$-regular class depends only on
the multiplicities and degrees of the irreducible factors of its
characteristic \pol .

\section{Proofs of Theorem~\ref{thm:main} and Corollary~\ref{cor:expform}}
\label{sec:computation}

We use Lemma~\ref{lemma:recur} in an inductive computation of the
generating functions \eqref{eq:genfct}. The next proposition gives the start of the induction.

\begin{prop}\label{prop:neq1}
  Suppose that $r=1$ or $n=1$.
  \begin{enumerate}
  \item  When $r=1$, $-\rchi_1(\GU nq)=\delta_{1,n}$ is $1$ for $n=1$
    and $0$ for all $n>1$. \label{prop:neq11}
  \item When $n=1$, $-\rchi_r(\GU 1q) =(q+1)^{r-1}$ for all $r \geq 1$.
   \label{prop:req1}
  \end{enumerate}
\end{prop}
\begin{proof}
    When $n=1$,
  $\Li_1^-(\F_q)^*=\emptyset$ is empty. Since $\rchi(\emptyset)=-1$, the
  $r$th equivariant \Euc\ is
  \begin{equation*}
    -\rchi_r(\GU 1q) = |\Hom{\Z^r}{\GU 1q}|/|\GU 1q| = |\GU 1q|^{r-1} =
    (q+1)^{r-1}  
  \end{equation*}
  for all $r \geq 1$.
  
  The first equivariant reduced \Euc\ of $\rchi_1(\gu)$, where $n>1$,
  is the classical \Euc\ of the orbit space
  $\operatorname{B}\!\Li_n^-(\F_q)^*/\gu$ for the $\gu$-action on the
  building, the classifying space of the poset $\Li_n^-(\F_q)^*$
  \cite[Proposition 2.13]{jmm:eulercent}. According to Quillen we can
  replace $\Li_n^-(\F_q)^*$ by the Brown poset $\cat S{\gu}{s+*}$ of
  nontrivial $s$-subgroups of $\gu$ \cite[Theorem 3.1]{quillen78}.  Webb's theorem
  \cite[Proposition~8.2.(i)]{webb87} applies to this replacement
  showing
  $\rchi_1(\GU nq) =
  \rchi\big(\operatorname{B}\!\Li_n^-(\F_q)^*/\gu\big) = 0$.
\end{proof}


Lemma~\ref{lemma:recur} can be  reformulated succinctly as the recurrence 
\begin{equation}
  \label{eq:FGL-recur}
  \FGL -{r+1}(q,x)  = T_{\SDIM -{}q}(\FGL -r(q,x)) T_{\SDIM +{}q}(\FGL +r(q^2,x^2)) 
\end{equation}
using the power series transform from
\cite[Definition~3.1]{jmm:eulergl+} reviewed in
Section~\ref{sec:mult-t_a-transf} below.

\begin{cor}\label{cor:A1q}
  The following identities hold
  \begin{gather*}
    T_{\SDIM -{}q}(1-x) T_{\SDIM +{}q}(1-x^2) = \frac{1-qx}{1+x}
    \qquad
    T_{\SDIM -{}q}(1+x) T_{\SDIM +{}q}(1-x^2) = \frac{1+qx}{1-x}
    \\
    T_{\SDIM -{}q} \Big( \frac{1+x}{1-x} \Big) = \frac{(1+x)(1+qx)}{(1-x)(1-qx)}   
  \end{gather*}
\end{cor}
\begin{proof}
  For the first identity, note that
  \begin{equation*}
    T_{\SDIM -{}q}(1-x)^{-1} T_{\SDIM +{}q}(1-x^2)^{-1} =
    1+\sum_{n \geq 1} (q^n+q^{n-1})x^n =
    1+\sum_{n \geq 1} (qx)^n +x\sum_{n \geq 1} (qx)^{n-1} = \frac{1+x}{1-qx} 
  \end{equation*}
  since the coefficient of $x^n$ in this power series is the number of
  self-dual monic \pol s in $\F_{q^2}[x]$ determined in
  Corollary~\ref{cor:sdcount}. (An alternative proof,
  \begin{multline*}
    T_{\SDIM -{}q}(1-x) T_{\SDIM +{}q}(1-x^2) =
    \prod_{d \geq 1} (1-x^d)^{\SDIM -dq} \prod_{d \geq 1}
    (1-x^{2d})^{\SDIM +dq} \\ =
    (1-x)^2
    \prod_{\substack{d \geq 1\\ \text{$d$ odd}}} (1-x^d)^{\IM  dq}
     (1-x^2)^{-1} \prod_{\substack{d \geq 2\\ \text{$d$ even}}}
     (1-x^d)^{\IM  dq} =
     \frac{1-x}{1+x} \prod_{d \geq 1} (1-x^d)^{\IM  dq}
     \stackrel{\text{\eqref{eq:Adq2}}}{=}
     \frac{1-qx}{1+x}
  \end{multline*}
  follows from Corollary~\ref{cor:Apmq}.)  Since $\SDIM -dq$ is
  nonzero only for odd $d$ (Proposition~\ref{prop:oddm}),
\begin{equation*}
  T_{\SDIM -{}q}(1+x) = \prod_{d \geq 1} (1+x^d)^{\SDIM -dq}
  = \prod_{d \geq 1} (1-(-x)^d)^{\SDIM -dq}
\end{equation*}
is the $\SDIM -{}q$-transform of $1-x$ evaluated at $-x$. Obviously, the
$\SDIM +{}q$-transform of $1-x^2$ is an even function of $x$. Thus
$T_{\SDIM -{}q}(1+x) T_{\SDIM +{}q}(1-x^2)$ is $T_{\SDIM -{}q}(1-x)
T_{\SDIM +{}q}(1-x^2)$ evaluated at $-x$. This proves the second
identity. The third identity is simply the quotient of the first two.
\end{proof}

 \begin{proof}[Proof of Theorem~\ref{thm:main}]
  The first generating function \eqref{eq:FGLpm}  is
  $\operatorname{FGL}^-_1(q,x)=1+x$ by
  Proposition~\ref{prop:neq1}.\eqref{prop:neq11}.
  Assume the formula of Theorem~\ref{thm:main} holds for some $r \geq 1$.
  Using a consequence of Corollary~\ref{cor:A1q},
  \begin{equation*}
    T_{\SDIM -{}q}(1 \pm q^j x) T_{\SDIM +{}q}(1-q^{2j}x^2) = \frac{1 \pm q^{j+1}x}{1 \mp q^jx}
  \end{equation*}
  which follows from the multiplicative property of these power series transforms \cite[(3.2)]{jmm:eulergl+},
  and recursion~\eqref{eq:FGL-recur}, the computation
  \begin{align*}
    \FGL -{r+1}(q,x) &=
    T_{\SDIM -{}q}(\FGL -r(q,x)) T_{\SDIM +{}q}(\FGL +r(q^2,x^2)) \\& =
    \frac{T_{\SDIM -{}q}(\prod_{j \equiv r \bmod 2}(1+q^jx)^{\binom rj})}
    {T_{\SDIM -{}q}(\prod_{j \not\equiv r \bmod 2}(1-q^jx)^{\binom rj})}
    \frac{T_{\SDIM +{}q}(\prod_{j \equiv r \bmod 2}(1-q^{2j}x^2)^{\binom rj})}
    {T_{\SDIM +{}q}(\prod_{j \not\equiv r \bmod 2}(1-q^{2j}x^2)^{\binom rj})} \\ &=
    \frac{\prod_{j \equiv r \bmod 2}(1+q^{j+1}x)^{\binom rj}}{\prod_{j \equiv r \bmod 2}(1-q^{j}x)^{\binom rj}}
    \frac{\prod_{j \not\equiv r \bmod 2}(1+q^{j}x)^{\binom rj}}{\prod_{j \not\equiv r \bmod 2}(1-q^{j+1}x)^{\binom rj}} \\ &=
    \frac{\prod_{j \equiv r+1 \bmod 2}(1+q^{j}x)^{\binom r{j-1}}}{\prod_{j \not\equiv r+1 \bmod 2}(1-q^{j}x)^{\binom rj}}
    \frac{\prod_{j \equiv r+1 \bmod 2}(1+q^{j}x)^{\binom rj}}{\prod_{j \not\equiv r+1 \bmod 2}(1-q^{j}x)^{\binom r{j-1}}} = 
    \frac{\prod_{j \equiv r+1 \bmod 2}(1+q^{j}x)^{\binom {r+1}j}}{\prod_{j \not\equiv r+1 \bmod 2}(1-q^{j}x)^{\binom {r+1}j}}
  \end{align*}
  shows that the formula holds also for $r+1$.
 \end{proof}

 \begin{proof}[Proof of Corollary~\ref{cor:expform}]
   The logarithm of the $(r+1)$th generating function $\FGL
   -{r+1}(q,x)$ is
\begin{multline*}
    \log \FGL -{r+1}(q,x) =
    \sum_{0 \leq j \leq r} (-1)^{r-j} \binom{r}{j} \log(1+(-1)^{r-j}q^jx)
    =
    \sum_{0 \leq j \leq r} (-1)^{r-j} \binom{r}{j} \sum_{n \geq 1}
    (-1)^{n+1}(-1)^{n(r-j)}q^{nj} \frac{x^n}{n} \\=
    \sum_{n \geq 1} (-1)^{n+1}  \sum_{0 \leq j \leq r}
    \binom{r}{j}(-1)^{(n+1)(r-j)} q^{nj} \frac{x^n}{n} =
    \sum_{n \geq 1} (-1)^{n+1} (q^n+(-1)^{n+1})^r \frac{x^n}{n}
    =
    -\sum_{n \geq 1} (-1)^{n} (q^n-(-1)^{n})^r \frac{x^n}{n}
  \end{multline*}   
 \end{proof}

\begin{rmk}[The Kn\"orr-Robinson conjecture]\label{rmk:KRC}
  The (non-block-wise form of the)
    the Kn\"orr-Robinson conjecture for the general unitary group $\GU nq$
    relative to the characteristic $s$ of $\F_q$ 
asserts that \cite{knorr_robinson:89,
  thevenaz93Alperin} \cite[Theorem 3.1]{quillen78}
\begin{equation*}
 -\rchi_2(\GU nq) = z_s(\GU nq)
\end{equation*}
where 
$z_s(\GU nq) = | \{ \chi \in \Irr{\C}{\GU nq} \mid |\GU nq|_s
\mid \chi(1)\}| $ is the number of irreducible complex representations
of $\GU nq$ of $s$-defect $0$ \cite[p 134]{isaacs}. As
$\FGL -2(q,x)=\frac{1+qx}{1-x}$, the left side is $q+1$ and so
is the right side \cite[Remark p 69]{humphreys06}. This confirms the
Kn\"orr--Robinson conjecture for $\GU nq$ relative to the defining
characteristic.
\end{rmk}

\subsection{Alternative presentations of the equivariant reduced \Euc s}
\label{sec:power-seri-expans-1}
The binomial formula applied to the right hand side of
Theorem~\ref{thm:main} gives the more direct expression
\begin{equation}\label{eq:prodform}
 -\rchi_{r+1}(\gu) = \sum_{n_0+\cdots+n_r=n} \prod_{0 \leq j \leq r}
    (-1)^{jn_j} \binom{(-1)^{j}\binom{r}{j}}{n_j}q^{n_j(r-j)}
  \end{equation}
  where the sum ranges over all $\binom{n+r}{n}$ weak compositions of
  $n$ into $r+1$ parts \cite[p 15]{stanley97}. This is also a consequence of \cite[Corollary~3.10]{jmm:eulergl+} and
\lq Ennola duality\rq ,
  \begin{equation}
  \label{eq:FGLpm}
  \FGL -{r}(q,x) =
  \FGL +{r}(-q,(-1)^{r}x), \qquad r \geq 1
\end{equation}
which follows by comparing the expressions of \cite[Theorem~1.4]{jmm:eulergl+} and Theorem~\ref{thm:main}.

We shall next relate
the equivariant \Euc s more directly to the structure of the general linear and unitary groups.
Recall that a (finite) multiset $\lambda$ is a (finite) base set $B(\lambda)$ with a
  multiplicity function assigning a natural number $E(\lambda,b)$  to all
  $b \in B(\lambda)$. Representing the multiset as
  $\lambda=\{b^{E(\lambda,b)} \mid b \in B(\lambda) \}$ and assuming the base  $B(\lambda)$ consists of natural numbers, we let
\begin{alignat*}{4}
  &|\lambda| = \sum_{b \in B(\lambda)} E(\lambda,b)&&\qquad 
  &&n(\lambda)=\sum_{b \in B(\lambda)} bE(\lambda,b) \\
  &T(\lambda) = \frac{n(\lambda)!}{\prod_{b \in B(\lambda)} E(\lambda,b)!
    b^{E(\lambda,b)}} &&\qquad
  &&U(\lambda,q) 
  = \prod_{b \in B(\lambda)} (q^b-1)^{E(\lambda,b)}  
\end{alignat*}
so that $|\lambda|$ is the cardinality or number of parts of $\lambda$, $\lambda$
partitions $n$, $\lambda \vdash n$, if $n(\lambda)=n$,
$T(\lambda)$ is the number of elements in the symmetric group
$\Sigma_{n(\lambda)}$ of cycle type $\lambda$ \cite[Proposition
1.1.1]{sagan:symmetric}, and $U(\lambda,q)$ is an integral \pol\ in $q$. With this notation,
the coefficients of $x^n$ in the power series of \cite[Corollary~1.5]{jmm:eulergl+} and
Corollary~\ref{cor:expform} are
\begin{equation}\label{eq:rchigl+}
  \rchi_{r+1}(\GL nq) =
  \frac{1}{n!} \sum_{ \lambda \vdash n} (-1)^{|\lambda|} T(\lambda) U(\lambda,q)^r, \quad
  -\rchi_{r+1}(\GU nq) = (-1)^{n(r+1)}
  \frac{1}{n!} \sum_{ \lambda \vdash n} (-1)^{|\lambda|} T(\lambda) U(\lambda,-q)^r
\end{equation}
with summation over all partitions $\lambda$ of $n$. 

Let $F_q$ denote the standard Frobenius endo\m\ of the algebraic group $\operatorname{GL}_n(\overline \F_s)$, $s=\operatorname{char}(\F_q)$, with fixed points $\operatorname{GL}_n(\overline \F_s)^{F_q} = \GL nq$.  The standard maximal torus $T_n(\overline{\F}_s) \cong \overline{\F}_s^\times \times \cdots \times \overline{\F}_s^\times $ consisting of the diagonal matrices in $\operatorname{GL}_n(\overline \F_s)$ is maximally split with respect to $F_q$ \cite[Definition~21.13, Example~21.14]{malle-testerman2011}. The Weyl group $W_n$ of $T_n(\overline{\F}_s)$ acts as the standard permutation representation of the symmetric group $\Sigma_n$ in the $n$-dimensional real vector space $X(T_n(\overline{\F}_s)) \otimes \R$ spanned by the character group $X(T_n(\overline{\F}_s))$.  As usual, $T_n(\overline{\F}_s)_w$ denotes the $F_q$-stable maximal torus of $\operatorname{GL}_n(\overline \F_s)$ corresponding to the Weyl group element $w \in W_n$ \cite[Proposition~25.1]{malle-testerman2011}.

Let $\sigma$ be the graph auto\m\ of $\operatorname{GL}_n(\overline \F_s)$ given by $\sigma(M) = A^{-1}(M^t)^{-1}A$,
$ M \in \operatorname{GL}_n(\overline \F_s)$,
where $A$ is the involutory permutation $A(i)=n+1-i$, $1 \leq i \leq n$.  The fixed points for the Steinberg endo\m\ $F_q\sigma$ are $\operatorname{GL}_n(\overline \F_s)^{F_q\sigma} = \GU nq$, $T_n(\overline{\F}_s)$ is a maximally split torus also with respect to $F_q\sigma$, and $\sigma$ acts on $X(T_n(\overline{\F}_s)) \otimes \R$ as $-A$ \cite[Examples~21.14.(2), 22.11.(2)]{malle-testerman2011}.

\begin{prop}\label{prop:altglpm}  
  The equivariant \Euc s of the $\operatorname{GL}_n^\pm(\F_q)$-posets $\Li_n^\pm(\F_q)^*$, $n \geq 1$, are
  \begin{align*}
    &\rchi_{r+1}( \GL nq) =\frac{(-1)^n}{|W_n|} \sum_{w \in W_n} \det(w) |T_n(\overline{\F}_s)_w^{F_q}|^r
      = \frac{(-1)^n}{|W_n|} \sum_{w \in W_n} \det(w) \det(q-w)^r \\
    -&\rchi_{r+1}(\GU nq) =
   \frac{(-1)^{\binom n2}}{|W_n|} \sum_{w \in W_n} \det(w) |T_n(\overline{\F}_s)_w^{F_q\sigma}|^r =
      \frac{1}{|W_n|} \sum_{w \in W_n} \det(w) \det(q + w)^r 
  \end{align*}
\end{prop}
\begin{proof}
  The number of elements of $T_n(\overline{\F}_s)_w$ that are fixed by the Frobenius endo\m\ $F_q$ is
\begin{equation*}
  |T_n(\overline{\F}_s)_w^{F_q}| = U(\lambda(w),q) = \det (q-w^{-1})
\end{equation*}
where $\lambda(w)$ is the cycle type of the permutation $w$ and determinants are computed in the real vector space $X(T_n(\overline{\F}_s)) \otimes \R$ \cite[Proposition~25.3, Example~25.4]{malle-testerman2011}. Equation~\eqref{eq:rchigl+} now takes the form
\begin{multline*}
  \rchi_{r+1}(\GL nq) 
  =\frac{(-1)^n}{|W_n|} \sum_{w \in W_n} \det(w) U(\lambda(w),q)^r
  \\ =\frac{(-1)^n}{|W_n|} \sum_{w \in W_n} \det(w) |T_n(\overline{\F}_s)_w^{F_q}|^r
  = \frac{(-1)^n}{|W_n|} \sum_{w \in W_n} \det(w) \det(q-w^{-1})^r
  = \frac{(-1)^n}{|W_n|} \sum_{w \in W_n} \det(w) \det(q-w)^r
\end{multline*}
since $(-1)^{\lambda(w)} = (-1)^n \det(w)$ and $\det(w)=\det(w^{-1})$ for all $w \in W_n$.

The number of elements of $T_n(\overline{\F}_s)_w$ that are fixed by Steinberg endo\m\ $F_q\sigma$ is \cite[Proposition~25.3.(c)]{malle-testerman2011}
\begin{equation*}
  |T_n(\overline{\F}_s)_w^{F_q\sigma}|  = \det (q-(-wA)^{-1}) = (-1)^n \det (-q-(wA)^{-1}) = (-1)^n U(\lambda(wA),-q)
\end{equation*}
Using Ennola duality \eqref{eq:FGLpm} combined with \eqref{eq:rchigl+}, and \cite[Proposition~25.3]{malle-testerman2011}, the calculation
\begin{multline*}
  -\rchi_{r+1}(\GU nq) =
   (-1)^{n(r+1)}\frac{(-1)^{n}}{|W_n|} \sum_{w \in W_n} \det(w) U(\lambda(w),-q)^r =
   \frac{(-1)^{nr}}{|W_n|} \sum_{w \in W_n} \det(w) U(\lambda(w),-q)^r \\=
   \frac{1}{|W_n|} \sum_{w \in W_n} \det(w) ((-1)^nU(\lambda(w),-q))^r =
 \frac{1}{|W_n|} \sum_{w \in W_n} \det(wA) ((-1)^nU(\lambda(wA),-q))^r \\ =
   \frac{1}{|W_n|} \sum_{w \in W_n} \det(wA) |T_n(\overline{\F}_s)_w^{F_q\sigma}|^r =
   \frac{1}{|W_n|} \sum_{w \in W_n} \det(wA) \det(q-(-wA)^{-1})^r \\=
   \frac{1}{|W_n|} \sum_{w \in W_n} \det(wA) \det(q + (wA)^{-1})^r =
   \frac{1}{|W_n|} \sum_{w \in W_n} \det(w) \det(q + w^{-1})^r =
   \frac{1}{|W_n|} \sum_{w \in W_n} \det(w) \det(q + w)^r 
\end{multline*}
finishes the proof. Here, $\det(wA) = \det(A) \det(w)$ where
$\det(A) = (-1)^{\binom n2}$ is $+1$ for $n \equiv 0,1 \bmod 4$ and $-1$ for $n \equiv 2,3 \bmod 4$. 
\end{proof}

Let $\rchi_{r+1}(\operatorname{GL}_n^\pm(\F_q))^{-1}$ denote the coefficient of $x^n$ in the reciprocal power series $\FGL{\pm}{r+1}(q,x)^{-1}$. The proof of the next result is similar to that of Proposition~\ref{prop:altglpm} except that it is based on the identities 
\begin{equation}
  \label{eq:rchigl+inv}
  \rchi_{r+1}(\GL nq)^{-1} =
  \frac{1}{n!} \sum_{ \lambda \vdash n}  T(\lambda) U(\lambda,q)^r, \qquad
  \rchi_{r+1}(\GU nq)^{-1} =
  \frac{(-1)^n}{n!} \sum_{ \lambda \vdash n}  T(\lambda) ((-1)^nU(\lambda,-q))^r,
\end{equation}
rather than \eqref{eq:rchigl+}.  The right hand sides of these identities are the
coefficients of $x^n$ in the reciprocal of the power series of
Corollary~\ref{cor:expform} and \cite[Corollary~1.5]{jmm:eulergl+}.

\begin{prop}\label{prop:altglpminv}  
  The reciprocal equivariant \Euc s of the $\operatorname{GL}_n^\pm(\F_q)$-posets $\Li_n^\pm(\F_q)^*$, $n \geq 1$,
  are
  \begin{align*}
    &\rchi_{r+1}( \GL nq)^{-1} =\frac{1}{|W_n|} \sum_{w \in W_n} |T_n(\overline{\F}_s)_w^{F_q}|^r
      = \frac{1}{|W_n|} \sum_{w \in W_n}  \det(q-w)^r \\
    (-1)^n&\rchi_{r+1}( \GU nq)^{-1} =
   \frac{1}{|W_n|} \sum_{w \in W_n}  |T_n(\overline{\F}_s)_w^{F_q\sigma}|^r =
      \frac{1}{|W_n|} \sum_{w \in W_n} \det(q + w)^r 
  \end{align*}
\end{prop}

Again, the case $r=1$ has special significance in that
$\rchi_2(\GL nq)^{-1}$ is the number of semisimple classes in $\GL nq$ and
$(-1)^n\rchi_2(\GU nq)^{-1}$ the number of semisimple classes in $\GU nq$  \cite[Proposition~3.7.4]{carter:finiteLie}.

\begin{exmp}\label{exmp:polid}
  The \pol\ identities 
\begin{alignat*}{3}
  &\frac{1}{|W_n|} \sum_{w \in W_n} \det(1-qw) = 1-q &&\qquad
  &&\frac{1}{|W_n|} \sum_{w \in W_n} \det(w) \det(q-w)^2 = (-1)^{n+1}n(q-1)^2q^{n-1} \\  
  &\frac{1}{|W_n|} \sum_{w \in W_n} \det(1+qw) = 1+q &&\qquad
 &&\frac{1}{|W_n|} \sum_{w \in W_n} \det(w) \det(q+w)^2 = n(q+1)^2q^{n-1} 
\end{alignat*}
are the instances $r=1,2$ of Proposition~\ref{prop:altglpm}. The right hand sides of the equations in the left column, where $r=1$, are the negative of the number of irreducible complex representations of $s$-defect $0$. (Indeed,
$|W_n|(1+q) = -|W_n|\rchi_2(\GU nq) = \sum_{w \in W_n} \det(w) \det(q+w) = \sum_{w \in W_n} \det(w^{-1}) \det(q+w^{-1}) = \sum_{w \in W_n} \det(w) \det(q+w^{-1}) = \sum_{w \in W_n} \det(1+qw)$.)

  The \pol\ identities
\begin{alignat*}{3}
  &\frac{1}{|W_n|} \sum_{w \in W_n} \det(q-w) = q^n-q^{n-1} &&\qquad
  && \frac{1}{|W_n|} \sum_{w \in W_n} \det(q-w)^2 = \frac{q-1}{q+1}(q^{2n}-1) \\
  &\frac{1}{|W_n|} \sum_{w \in W_n} \det(q+w) = q^n+q^{n-1} &&\qquad
  && \frac{1}{|W_n|} \sum_{w \in W_n} \det(q+w)^2 = \frac{q+1}{q-1}(q^{2n}-1)
\end{alignat*}
are the instances $r=1,2$ of Proposition~\ref{prop:altglpminv}.  The right hand sides of the equations in the left column, where $r=1$, count semisimple classes.
\end{exmp}

The next corollary, an immediate consequence of \eqref{eq:rchigl+} and
Proposition~\ref{prop:altglpm}, lists the generating functions for the equivariant \Euc s $\rchi_{r+1}(\mathrm{GL}^{\pm}(n,\F_q))$, $r \geq 0$,  for a  fixed $n$. (The first part is \cite[Proposition~4.19]{jmm:eulergl+}.)

\begin{cor}\label{prop:GGU}
  For any fixed $n \geq 1$,
  \begin{align*}
    &\sum_{r \geq 0} \rchi_{r+1}( \GL nq)x^r =
      \frac{1}{n!} \sum_{ \lambda \vdash n}  (-1)^{|\lambda|}  \frac{T(\lambda)}{1- U(\lambda,q)x} =
      \frac{(-1)^n}{|W_n|} \sum_{w \in W_n} \frac{\det(w)}{1-x\det(q-w)} \\
    &\sum_{r \geq 0} -\rchi_{r+1}( \GU nq)x^r =
      \frac{(-1)^n}{n!} \sum_{ \lambda \vdash n}  (-1)^{|\lambda|}  \frac{T(\lambda)}{1-(-1)^n U(\lambda,-q)x} =
    \frac{1}{|W_n|} \sum_{w \in W_n} \frac{\det(w)}{1-x\det(q+w)}
  \end{align*}
  \end{cor}

For example, the power series $n!\sum\limits_{r \geq 0} -\rchi_{r+1}( \GU nq)x^r$ is
\begin{equation*}
  \frac{1}{1-(q+1)x}, \qquad
                 \frac{1}{1-(q+1)^2x} - \frac{1}{1-(q^2-1)x}, \qquad
  \frac{1}{1-(q+1)^3x}
  -\frac{3}{1-(q^2-1)(q+1)x}
  +\frac{2}{1-(q^3+1)x}
\end{equation*}
for $n=1,2,3$.

\subsection{Hasse--Weil zeta functions and equivariant \Euc s}
\label{sec:conn-hasse-weil}

 The Hasse--Weil zeta function for a
projective variety $V$ defined over $\F_q$,
\begin{equation*}
  Z(V/\F_q;T) = \exp \Big( \sum_{n \geq 1}|V(\F_{q^n})| \frac{T^n}{n} \Big)
\end{equation*}
encodes the number of points on $V$ over $\F_{q^n}$ for all $n \geq 1$
\cite[V.2--V.3]{silverman2009}.


\begin{prop}\label{prop:halle}
  For any $m \geq 1$
  \begin{equation*}
    \FGL -{2m+1}(q,-T) = Z(E^m/\F_{q^2};T)^{-1} 
  \end{equation*}
  is the reciprocal of the Hasse--Weil zeta function of the $m$-fold self-product
  $E^m=E \times \cdots \times E$ of any supersingular elliptic curve $E$
  defined over $\F_{q^2}$.
\end{prop}
\begin{proof}
  Let $E$ be any supersingular elliptic curve defined over $\F_{q^2}$
  \cite[Definition, p.\ 145]{silverman2009}. We note that
  \begin{equation}\label{eq:Z(E)}
     Z(E/\F_{q^2};T) = \frac{(1+qT)^2}{(1-T)(1-q^2T)} = \FGL
     -3(q,-T)^{-1} \stackrel{\text{Cor.~\ref{cor:expform}}}{=}
     \exp\Big( \sum_{n \geq 1}   (q^n-(-1)^n)^2 \frac{T^n}{n} \Big)
   \end{equation}
  and hence
  \begin{equation*}
     Z(E^m/\F_{q^2};T) = \exp\Big( \sum_{n \geq 1} (q^n-(-1)^n)^{2m}
     \frac{T^n}{n} \Big) \stackrel{\text{Cor.~\ref{cor:expform}}}{=}
     \FGL -{2m+1}(q,-T)^{-1}
  \end{equation*}
  as $|E(\F_{q^2})| = (q^n-(-1)^n)^2$ by \eqref{eq:Z(E)} and
  $E^m(\F_{q^{2n}}) = E(\F_{q^{2n}})^m$ for general reasons.
  \end{proof}

\section{Transforms of \pol\ power series and \pol\ identities}
\label{sec:mult-t_a-transf}

Let $F(q,x) = 1 + \sum_{n \geq 1} a(n)(q)x^n \in 1 +(x) \subseteq \Q[q][[x]]$ be a power series with leading term $1$ in the power series ring over the ring of rational \pol s in $q$.
 Given a sequence $S=(S(n)(q))_{n \geq 1}$ of rational numbers defined for each prime power $q$,  the $S$-transform of $F(q,x)$ is the power series \cite[Definition 3.1]{jmm:eulergl+}
 \begin{equation*}  
      T_S(F(q,x)) = \prod_{d \geq 1} F(q^d,x^d)^{S(d)(q)}
    \end{equation*}
The transformation
   $T_S \colon 1+(x) \to 1+(x)$ is multiplicative in $F$ and exponential in $S$ the sense that
   \begin{equation*}
     T_S(1)=1, \qquad T_S(F_1(q,x) F_2(q,x)) = T_S(F_1(q,x) T_S(F_2(q,x)), \qquad T_{mS}(F(q,x)) = T_S(F(q,x))^m
   \end{equation*}
  for all $F_1(q,x), F_2(q,x) \in 1 + (x) \subseteq \Q[q][[x]]$ and rational numbers $m \in \Q$ \cite[\S3.2]{jmm:eulergl+}.

  For example, the $S$-transform of $1 \pm x^k$ is easily determined by evaluating the  
  coefficient of $x^{kn}$ in the infinite product expansion $T_S(1 \pm x^k) = \prod_{d \geq 1} (1 \pm x^{kd})^{S(d)(q)}$. (See
 the beginning of Subsection~\ref{sec:power-seri-expans-1} for multiset notation. We use the convention that the binomial coefficient $\binom m{-k} = (-1)^k \binom mk$ for all natural numbers $k$.)


  \begin{lemma}\label{lemma:basic2}
    For any rational number $m$ and natural number $k$,
  the $mS$-transform of the power series $1 \pm x^k$ is
  \begin{equation*}
    T_{mS}(1 \pm x^k) = 1 + \sum_{n \geq 1} \Big[
    \sum_{\lambda \vdash n} \prod_{d \in B(\lambda)}
    \binom{m  S(d)(q)}{\pm E(\lambda,d)} \Big] x^{kn}
  \end{equation*}
\end{lemma}

The below corollary is Lemma~\ref{lemma:basic2} applied to the classical identity $T_{\IM {}q}(1-x) = \frac{1-qx}{1-x}$ \eqref{eq:Adq2}, while the theorem is the lemma applied to the identities
\begin{equation*}
  T_{a_{r+1}^{\pm}(q)}(1-x) = \FGL{\pm}{r+1}(q,x), \qquad
  a_{r+1}^{\pm}(q,n) = \frac{1}{n} \sum_{d \mid n} (\pm 1)^d \mu(n/d) (q^d - (\pm 1)^d)^r
\end{equation*}
found below Theorem~\ref{thm:mainprim} or below \cite[Theorem~1.7]{jmm:eulergl+}, and to the power series identities of Corollary~\ref{cor:A1q}.

\begin{cor}\label{cor:polidgl} 
  For any rational number $m$,
  \begin{equation*}
    1 + \sum_{n \geq 1} \Big[
    \sum_{\lambda  \vdash n} \prod_{d \in B(\lambda)}
    \binom{m\IM dq}{-E(\lambda,d)} \Big] x^n
    = \left( \frac{1-qx}{1-x} \right)^m
     \end{equation*}
\end{cor}

\begin{thm}\label{thm:guglpolid}  
  For any rational number $m$ and natural number $r \geq 0$
  \begin{gather*}
    1 + \sum_{n \geq 1} \Big[
    \sum_{\lambda  \vdash n} \prod_{d \in B(\lambda)}
    \binom{ma_{r+1}^\pm(q,d)}{-E(\lambda,d)} \Big] x^n
    = \FGL{\pm}{r+1}(q,x)^m \\
    \Big(1+ \sum_{n^- \geq 1}\Big[\sum_{\lambda^- \vdash n^-} \prod_{d^- \in B(\lambda^-)}
    \binom{ m \SDIM -{d^-}q}{-E(\lambda^-,d^-)}\Big]x^{n^-}\Big)
    \Big(1 + \Big[\sum_{\lambda^+ \vdash n^+} \prod_{d^+ \in B(\lambda^+)}
    \binom{ m \SDIM +{d^+}q}{-E(\lambda^+,d^+)}\Big] x^{2n^+} \Big) =
    \Big(\frac{1-qx}{1+x}\Big)^m  \\
     \Big(1+ \sum_{n^- \geq 1}\Big[\sum_{\lambda^- \vdash n^-} \prod_{d^- \in B(\lambda^-)}
    \binom{ m \SDIM -{d^-}q}{-E(\lambda^-,d^-)}\Big]x^{n^-}\Big)
    \Big(1 + \Big[\sum_{\lambda^+ \vdash n^+} \prod_{d^+ \in B(\lambda^+)}
    \binom{ -m \SDIM -{d^+}q}{E(\lambda^+,d^+)}\Big] x^{n^+} \Big) =
    \Big(\frac{(1-qx)(1-x)}{(1+x)(1+qx)}\Big)^m
  \end{gather*}
\end{thm}


Th\'evenaz' \pol\ identities for partitions \cite[Theorem A, Theorem B]{thevenaz92poly} are the cases $m=\pm 1$ of Corollary~\ref{cor:polidgl}. The purely combinatorial proof of a generalised version of Th\'evenaz' \pol\ identities
presented here
may qualify as an answer to question (1) on p.\ 129 of \cite{thevenaz92poly}.  Corollary~\ref{cor:polidgl} is the special case $r=1$ of the first equation of Theorem~\ref{thm:guglpolid} as $a_2^+(q,d) = \IM dq$.

These \pol\ identities are examples of
Corollary~\ref{cor:polidgl} and Theorem~\ref{thm:guglpolid}  at $n=3$
 \begin{gather*}
  \binom{m\IM 3q}{-1} + \binom{m\IM 2q}{-1}\binom{m\IM 1q}{-1} + \binom{m\IM 1q}{-3} =
  \begin{cases}
    (q-1)q^2 & m=-1 \\
    \frac{1}{16}(q-1)(5q^2+2q+1) & m=-\frac{1}{2} \\
    \frac{1}{16}(1-q)(q^2+2q+5) & m=\frac{1}{2} \\
    1-q & m=1
  \end{cases} \\
   \binom{m \SDIM -3q}{-1} + \binom{m \SDIM -1q}{-3}
    + \binom{m \SDIM -1q}{-1} \binom{m \SDIM +1q}{-1} =
    \begin{cases}
      -q^3-q^2 & m = -1 \\
      \frac{1}{16}(q+1)(5q^2-2q+1) & m=-\frac{1}{2} \\
      -\frac{1}{16}(q+1)(q^2-2q+5) & m=\frac{1}{2} \\
       1+q & m = +1 \\ 
    \end{cases} \\
    \binom{m\SDIM -3q}{-1} + \binom{m\SDIM -1q}{-3} +
    \binom{m\SDIM -1q}{-2} \binom{-m\SDIM -1q}{1} + \binom{m\SDIM -1q}{-1}\binom{-m\SDIM -1q}{2} \\ +
    \binom{-m\SDIM -3q}{1} + \binom{-m\SDIM -1q}{3} =
    \begin{cases}
       \frac{1}{2}(q^3 + q^2 + q + 1) & m = -\frac{1}{2} \\
      2(q+1)(q^2+q+1) & m=-1 \\ 4(q+1)(3q^2+5q+3) & m=-2
    \end{cases}
  \end{gather*}
  The terms on the left side correspond to the three partitions $\{3^1\}, \{2^11^1\}, \{1^3\}$, of $3$ in the first,
 to $(n^-,n^+)$ in $\{(3,0),(1,1)\}$ in the second,  and to $(n^-,n^+)$ in $\{(3,0), (2,1), (1,2), (0,3)\}$ in the third example.


\section{Primary equivariant reduced \Euc s}
\label{sec:p-prim-equiv}

Let $p$ be a prime and, as in the previous sections, $q$ a prime
power. (The prime $p$ may or may not divide the prime power $q$.)  In
this section we discuss the $p$-primary equivariant reduced \Euc s of
the $\gu$-poset $\Li_n^-(\F_q)^*$.

\begin{defn}\cite[(1-5)]{tamanoi2001}\label{defn:pprimeuc}
  The $r$th $p$-primary equivariant reduced \Euc\ of the $\gu$-poset
  $\Li_n^-(\F_q)^*$ is the normalised sum
\begin{equation*}
  \rchi_r(p,\GU nq) = \frac{1}{|\gu|}
  \sum_{X \in \Hom{\Z \times
      \Z_p^{r-1}}{\gu}}\rchi(C_{\Li_n^-(\F_q)^*}(X(\Z \times \Z_p^{r-1})))
\end{equation*}
of reduced \Euc s of fixed sub-posets.
\end{defn}

In this definition, $\Z_p$ denotes the ring of $p$-adic
integers and the sum ranges over all
homo\m s of $\Z \times \Z_p^{r-1}$ into $\gu$ or, equivalently, over all
commuting $r$-tuples
$(X_1,X_2,\ldots,X_r)$ of elements of $\GU nq$ where
$X_2,\ldots,X_r$ have $p$-power order. The {\em first\/} $p$-primary
equivariant reduced \Euc\ is independent of $p$ and agrees with the
first equivariant reduced \Euc . If $p$ divides $q$, then
$\rchi_r(p,\GU nq)=0$ for all $r,n>1$ by
Lemma~\ref{lemma:contractability}.

The $r$th $p$-primary equivariant {\em unreduced\/} \Euc\
$\chi_r(p,\GU nq)$, obtained by replacing the reduced \Euc s with
\Euc s  in Definition~\ref{defn:pprimeuc}, agrees with the
\Euc\
computed in
Morava $K(r)$-theory at $p$
of the homotopy orbit space $\B\!\Li^-_n(\F_q)^*_{h\!\gu}$ for the action of
$\gu$ on the classifying space for the poset $\Li^-_n(\F_q)^*$
\cite{HKR2000}
\cite[2-3, 5-1]{tamanoi2001}
\cite[Remark 7.2]{jmm:partposet2017}.

The $r$th $p$-primary  generating function at $q$ is the integral power series
\begin{equation}\label{eq:primgenfct}
   \FGL{-}r(p,q,x) = 1- \sum_{ n \geq 1} \rchi_{r}(p,\GU nq) x^n \in \Z[[x]]
\end{equation}
associated to the sequence $(-\rchi_r(p,\GU nq))_{n \geq 1}$ of the  {\em 
  negative\/} of the $p$-primary equivariant reduced \Euc s. For $r=1$,
$\FGL{-}1(p,q,x)=\FGL -1(q,x) = 1+x$, and when $p \mid q$,
$\FGL{-}r(p,q,x)=1+x$ for all $r \geq 1$. The interesting case is when
the characteristic of $\F_q$ is different from $p$.

\begin{defn}\label{defn:Adpq2}
  For every integer $d \geq 1$,
  \begin{itemize}
  \item $\IM d{p,q}$ is the number of $p$-power order  Irreducible Monic \pol s of degree
    $d$ over $\F_{q}$ with nonzero constant term
  \item $\SDIM -d{p,q}$ is the number of $p$-power order Self-Dual
    Irreducible Monic \pol s of degree $d$ over $\F_{q^2}$ with
    nonzero constant term
  \item $\SDIM +d{p,q} =\frac{1}{2}(\IM d{p,q^2}- \SDIM -d{p,q})$ is the
    number of unordered pairs of $p$-power order non-self-dual
    irreducible monic \pol s of degree $d$ over $\F_{q^2}$ with
    nonzero constant term
  \end{itemize}
\end{defn}

The next lemma follows from Lemma~\ref{lemma:meyn} combined with the
fact from \cite[Lemma~3.6]{lidlnieder97} that $x^a-1$ divides $x^b-1$
in $\F_{q^2}[x]$ if and only if $a$ divides $b$.

\begin{lemma}\label{lemma:SDIMmmu}
  Assume $p \nmid q$ and let $m \geq 1$ be an odd integer.
  \begin{enumerate}
  \item The
  $p$-power order
  self-dual irreducible monic \pol s of degree dividing $m$ are precisely
  the irreducible factors of  $x^{(q^m+1)_p}-1 \in \F_{q^2}[x]$. \label{lemma:SDIMmmua}
\item $\sum\limits_{d \mid m}d \SDIM -d{p,q} = (q^m+1)_p$ and $m \SDIM -m{p,q} = \sum\limits_{d \mid m} \mu(m/d)(q^d+1)_p$. \label{lemma:SDIMmmub}
  \end{enumerate}
\end{lemma}

The $p$-primary version of Lemma~\ref{lemma:recur} states that for $p \nmid q$, $r \geq 1$, and $n>1$,
\begin{equation*}
    \rchi_{r+1}(p,\GU nq) = \sum_{[g] \in [\gu_p] }
    \rchi_{r}(p,C_{\Li_n^-(\F_q)^*}(g), C_{\gu}(g))
  \end{equation*}
  where the sum ranges over  the set
  $[\gu_p]$ of conjugacy classes of $p$-elements. The point here is that a semisimple element of $\gu$ has $p$-power order if and only all irreducible factors of its characteristic \pol\ have $p$-power order \cite[Lemma~4.4]{jmm:eulergl+}. In terms of generating functions we get the $p$-primary version
\begin{equation}
  \label{eq:FGL-recurprim}
  \FGL -{r+1}(p,q,x)  = T_{\SDIM -{}{p,q}}(\FGL -r(p,q,x)) T_{\SDIM
    +{}{p,q}}(\FGL +r(p,q^2,x^2)) 
\end{equation}
of \eqref{eq:FGL-recur}. In the following we prefer to work with the equivalent relation
\begin{equation}
  \label{eq:prodexpprim}
   a_{r+1}^-(p,q,N) =
\sum_{d \mid N} a_r^-(p,q^d,N/d)\SDIM -d{p,q} +
  \sum_{2d \mid N} a_r^+(p,q^{2d},N/2d)\SDIM +d{p,q}
\end{equation}
where
\begin{equation*}
  a_{r}^-(p,q,n) =
  \frac{1}{n} \sum_{d \mid n} (-1)^d\mu(n/d)(q^d -(-1)^d)^{r-1}_p
  \qquad
  a_{r}^+(p,q,n) =
  \frac{1}{n} \sum_{d \mid n} \mu(n/d)(q^d -1)^{r-1}_p
\end{equation*}
To go from \eqref{eq:FGL-recurprim} to \eqref{eq:prodexpprim} we use 
the infinite product expansions 
\begin{align*}
  &\FGL -{r+1}(p,q,x) =
                                                                     \prod_{N \geq 1} (1-x^N)^{a_{r+1}^-(p,q,N)}  \\
  &T_{\SDIM -{}{p,q}}(\FGL -r(p,q,x)) =
                        \prod_{d \geq 1} \FGL -r(p,q^d,x^d)^{\SDIM -d{p,q}}       =         
                        \prod_{n,d \geq 1} (1-x^{dn})^{a_{r}^-(p,q^d,n)
                        \SDIM -d{p,q}}  \\
  &T_{\SDIM +{}{p,q}} \FGL +r(p,q^2,x^2) =
                        \prod_{d \geq 1} \FGL +r(p,q^{2d},x^{2d})^{\SDIM +d{p,q}}       =                           
                        \prod_{n,d \geq 1} (1-x^{2dn})^{a_{r}^+(p,q^{2d},n)
                        \SDIM +d{p,q}}
\end{align*}
of the three factors in \eqref{eq:FGL-recurprim} obtained by applying \cite[Lemma~3.7]{jmm:eulergl+} to
the expressions of Theorem~\ref{thm:mainprim} and \cite[Theorem~1.7]{jmm:eulergl+}.

\begin{proof}[Proof of Theorem~\ref{thm:mainprim}]
We must show that the functions $a_r^\pm(p,q,n)$
satisfy
recurrence relation~\eqref{eq:prodexpprim}.
The right side of \eqref{eq:prodexpprim} multiplied by $N$ is
\begin{multline*}
  \sum_{d \mid N}  d \SDIM -d{p,q} \sum_{e \mid (N/d)} (-1)^e \mu(N/de) (q^{de}-(-1)^e)_p^{r-1}  +
  \sum_{2d \mid N}   2d\SDIM +d{p,q} \sum_{e \mid (N/2d)} \mu(N/2de) (q^{2de}-1)^{r-1}_p \\ =
  \sum_{d \mid N}  d \SDIM -d{p,q} \sum_{e \mid (N/d)} (-1)^e \mu(N/de) (q^{de}-(-1)^e)_p^{r-1}  -
  \sum_{2d \mid N}  d \SDIM -d{p,q}  \sum_{e \mid (N/2d)}  \mu(N/2de) (q^{2de}-1)_p^{r-1}  \\ +
  \sum_{2d \mid N} d \IM d{p,q^2} \sum_{e \mid (N/2d)}  \mu(N/2de) (q^{2de}-1)_p^{r-1} 
\end{multline*}
When $N$ is odd, we are left with
\begin{multline*}
  -\sum_{d \mid N}  d \SDIM -d{p,q} \sum_{e \mid (N/d)}  \mu(N/de) (q^{de}+1)_p^{r-1}  =
  -\sum_{f \mid d_1 \mid d_2 \mid N} \mu(d_1/f)\mu(N/d_2)(q^f+1)_p(q^{d_2}+1)^{r-1}_p \\=
  -\sum_{f \mid d_2 \mid N} \mu(N/d_2) (q^f+1)_p(q^{d_2}+1)^{r-1}_p \sum_{\{d_1 \colon f \mid d_1 \mid d_2\}} \mu(d_1/f) =
  -\sum_{d \mid N} \mu(N/d) (q^d+1)^r_p = Na_{r+1}^-(p,q,N)
\end{multline*}
where we first use Lemma~\ref{lemma:SDIMmmu}.\eqref{lemma:SDIMmmub} and next observe that
the sum
\begin{equation*}
  \sum_{\{d_1 \colon f \mid d_1 \mid d_2\}} \mu(d_1/f) =
  \begin{cases}
    1 & f=d_2 \\ 0 & f < d_2 
  \end{cases}
\end{equation*}
contributes only when $f=d_2$. Thus \eqref{eq:prodexpprim} holds under the assumption that $N$ is odd.

When $N=2N_1$ is even we have
\begin{multline*}
  \sum_{2d \mid N} d \IM d{p,q^2} \sum_{e \mid (N/2d)}  \mu(N/2de) (q^{2de}-1)_p^{r-1} =
  \sum_{d \mid N_1} d \IM d{p,q^2} \sum_{de \mid N_1} \mu(N_1/de)(q^{2de}-1)_p^{r-1} \\=
  \sum_{d_1 \mid d_2 \mid N_1} d_1 \IM {d_1}{p,q^2} \mu(N_1/d_2)(q^{2d_2}-1)_p^{r-1} =
  \sum_{f \mid d_2 \mid N} \mu(N/d_2) (q^{2f}-1)_p(q^{2d_2}-1)_p^{r-1} \sum_{\{d_1 \colon f \mid d_1 \mid d_2\}} \mu(d_1/f) \\ =
  \sum_{d_1 \mid N_1} \mu(N_1/d_1)(q^{2d_1}-1)_p^r =
  \sum_{2d\mid N} \mu(N/2d)(q^{2d}-1)_p^r 
\end{multline*}
which is the part of $Na_{r+1}^-(p,q,N)$ defined by the {\em even\/} divisors of $N$.
Remember that $\SDIM -dq$, and then also $\SDIM -d{p,q}$, is nonzero only for odd $d$ (Proposition~\ref{prop:oddm}).
Thus the claim for even $N=2N_1$ is 
\begin{multline*}
  -\sum_{\substack{d \mid N \\ \text{$d$ odd}}} \mu(N/d) (q^d+1)^r_p =\\ 
  \sum_{\substack{d_1 \mid d_2 \mid N\\ \text{$d_1$ odd}}} (-1)^{d_2/d_1} d_1 \SDIM -{d_1}{p,q} \mu(N/d_2) (q^{d_2}-(-1)^{d_2/d_1})^{r-1}_p  
  -\sum_{\substack{d_1 \mid d_2 \mid N_1 \\ \text{$d_1$ odd}}} d_1 \SDIM -{d_1}{p,q} \mu(N_1/d_2) (q^{2d_2}-1)^{r-1}_p
\end{multline*}
Note that for every $j=1,\ldots,k$, where $N_2 = 2^k$ is the highest power of $2$ dividing $N$, the part of the first sum with $2^j \parallel d_2$ is annihilated by the part of the second sum with $2^{j-1} \parallel d_2$. Thus the right hand side reduces to the
part of the first sum where $d_2$ is odd. By the computation just done for odd $N$, that sum equals the left hand side.
Thus \eqref{eq:prodexpprim} holds also for even $N$.
\end{proof}

When $p$ does not divide $q$, the sequences
$(\SDIM{\pm}d{p,q})_{d \geq 1}$ and hence the generating functions
$\FGL{-}{r+1}(p,q,x)$, $r \geq 1$, depend only on the closure
$\overline{\gen q}$ of the cyclic subgroup generated by $q$ in the
topological group $\Z_p^\times$ of $p$-adic units
\cite[Lemma~4.9]{jmm:eulergl+}.  For instance, the $2$-primary power series
$\FGL -{r+1}(2,q,x)$ are identical for $q=3,11,19,27,\ldots$,
with $\log \FGL -{r+1}(2,3,x) =  \sum_{n \geq 1} (-1)^{n+1}(4n)_2^r x^n/n$,
and the $3$-primary
power series $\FGL -{r+1}(3,q,x)$ are identical for
$q=2,5,11,23,\ldots$ with $\log \FGL -{r+1}(3,2,x) =  \sum_{n \geq 1} (-1)^{n+1}(3n)_3^r x^n/n$ \cite[Figure 3, Example~4.16]{jmm:eulergl+}.

\subsection{Alternative presentations of $p$-primary equivariant reduced \Euc s}
\label{sec:power-seri-expans}

It is immediate from \cite[Theorem~1.7]{jmm:eulergl+} and Theorem~\ref{thm:mainprim} that there is
\lq Ennola duality\rq\
  \begin{equation}
  \label{eq:FGLpm}
  \FGL -{r}(p,q,x) =
  \FGL +{r}(p,-q,(-1)^{r}x), \qquad r \geq 1
\end{equation}
between the $p$-primary generating functions for $\operatorname{GL}^{\pm}_n(\F_q)$.

We can now proceed exactly as in Subsection~\ref{sec:power-seri-expans-1} to prove
the next two propositions. In Proposition~\ref{prop:altglpminvpprim},
$\rchi_{r+1}(p,\operatorname{GL}_n^\pm(\F_q))^{-1}$ denotes the coefficient of $x^n$ in the reciprocal power series $\FGL{\pm}{r+1}(p,q,x)^{-1}$.

\begin{prop}\label{prop:altglpmpprim}  
  The $p$-primary equivariant \Euc s of the $\operatorname{GL}_n^\pm(\F_q)$-posets
  $\Li_n^\pm(\F_q)^*$, $n \geq 1$, are
  \begin{align*}
  &\rchi_{r+1}(p, \GL nq)
    = \frac{1}{n!} \sum_{ \lambda \vdash n} (-1)^{|\lambda|} T(\lambda) U(\lambda,q)^r_p 
    =\frac{(-1)^n}{|W_n|} \sum_{w \in W_n} \det(w) |T_n(\overline{\F}_s)_w^{F_q}|^r_p \\
    & \phantom{\rchi_{r+1}(p, \GL nq)} = \frac{(-1)^n}{|W_n|} \sum_{w \in W_n} \det(w) \det(q-w)^r_p \\
    -&\rchi_{r+1}(p,\GU nq) = (-1)^{n(r+1)}
     \frac{1}{n!} \sum_{ \lambda \vdash n} (-1)^{|\lambda|} T(\lambda) U(\lambda,-q)^r_p  =
       \frac{(-1)^{\binom n2}}{|W_n|} \sum_{w \in W_n} \det(w) |T_n(\overline{\F}_s)_w^{F_q\sigma}|^r_p \\
      & \phantom{\rchi_{r+1}(p,\GU nq)} = 
      \frac{1}{|W_n|} \sum_{w \in W_n} \det(w) \det(q + w)^r_p 
    \end{align*}
  \end{prop}

\begin{prop}\label{prop:altglpminvpprim}  
  The reciprocal $p$-primary equivariant \Euc s
  of the $\operatorname{GL}_n^\pm(\F_q)$-posets $\Li_n^\pm(\F_q)^*$, $n \geq 1$,
  are
  \begin{align*}
    &\rchi_{r+1}(p, \GL nq)^{-1} =\frac{1}{|W_n|} \sum_{w \in W_n} |T_n(\overline{\F}_s)_w^{F_q}|^r_p
      = \frac{1}{|W_n|} \sum_{w \in W_n}  \det(q-w)^r_p \\
    (-1)^n&\rchi_{r+1}(p, \GU nq)^{-1} =
   \frac{1}{|W_n|} \sum_{w \in W_n}  |T_n(\overline{\F}_s)_w^{F_q\sigma}|^r_p =
      \frac{1}{|W_n|} \sum_{w \in W_n} \det(q + w)^r_p 
  \end{align*}
\end{prop}

A slight
modification of \cite[Proposition~3.7.4]{carter:finiteLie} shows that
$(\pm 1)^n\rchi_2(p,\operatorname{GL}^\pm_n(\F_q))^{-1}$ 
equals the number of semisimple $p$-classes in $\operatorname{GL}^\pm_n(\F_q)$. 

The next corollary, the $p$-primary version of Corollary~\ref{prop:GGU},
is an immediate consequence of Proposition~\ref{prop:altglpmpprim} and it specifies the generating functions for the $p$-primary equivariant \Euc s expanded after the parameter $r$ and with fixed $n$.

\begin{cor}\label{prop:GGUpprim}
  For any fixed $n \geq 1$,
  \begin{align*}
    &\sum_{r \geq 0} \rchi_{r+1}(p, \GL nq)x^r =
      \frac{1}{n!} \sum_{ \lambda \vdash n}  \frac{ (-1)^{|\lambda|} T(\lambda)}{1- U(\lambda,q)_px} =
      \frac{(-1)^n}{|W_n|} \sum_{w \in W_n} \frac{\det(w)}{1-x\det(q-w)_p} \\
    &\sum_{r \geq 0} -\rchi_{r+1}(p, \GU nq)x^r =
      \frac{(-1)^n}{n!} \sum_{ \lambda \vdash n}  \frac{ (-1)^{|\lambda|} T(\lambda)}{1-(-1)^n U(\lambda,-q)_px} =
    \frac{1}{|W_n|} \sum_{w \in W_n} \frac{\det(w)}{1-x\det(q+w)_p}
  \end{align*}
  \end{cor}

  For example, when $n=3$, $p=2$, and $q=3,11,19,27, \ldots$ is any prime power with $q \equiv 3 \bmod 8$, $(q^2-1)_2=(3^2-1)_2$, 
  the generating function (times $3!$)  for the $2$-primary equivariant reduced \Euc s of $\GU 3q$ is
  \begin{equation*}
    3!\sum_{r \geq 0} -\rchi_{r+1}( 2,\GU 3q)x^r =
    \frac{1}{1-x(q+1)_2^3} - \frac{3}{1-x(q^2-1)_2(q+1)_2} + \frac{2}{1-x(q^3+1)_2} =
    \frac{1}{1-64x} + \frac{2}{1-4x} - \frac{3}{1-32x}
  \end{equation*}
 with the three terms corresponding to the three partitions $\{1^3\},\{1^12^1\},\{3^1\}$ of $3$.

\section*{Acknowledgments}
\label{Acknowledgments}
I with to thank Lars Halvard Halle for pointing out the
connection 
to Hasse--Weil zeta-functions, 
and Jean Michel for a helpful remark, used in Subsection~\ref{sec:power-seri-expans-1}, that, to him but not to me, was completely obvious.

\bibliographystyle{amsplain}
\bibliography{top}
\end{document}